\documentclass[review]{elsarticle}

\usepackage{amsmath}
\usepackage{amssymb}
\usepackage{amsthm}
\usepackage{lineno}
\usepackage{hyperref}
\modulolinenumbers[5]

\journal{Journal of \LaTeX\ Templates}

\newtheorem{Theorem}{Theorem}
\newtheorem{Definition}{Definition}
\newtheorem{Lemma}{Lemma}
\newtheorem{Remark}{Remark}

\allowdisplaybreaks

\begin{document}

\begin{frontmatter}

\title{Quadrilateral grid generation supported on complex 
        internal boundaries using spectral methods}



\author[mymainaddress]{Sa\'ul E. Buitrago Boret\corref{mycorrespondingauthor}}
\ead{sbutrago@usb.ve}

\author[mymainaddress]{Oswaldo J. Jim\'enez P.}
\cortext[mycorrespondingauthor]{Corresponding author}
\ead{oswjimenez@usb.ve}

\address[mymainaddress]{Dpto. C\'omputo Cient\'{\i}fico y Estad\'{\i}stica, Universidad Sim\'on
Bol\'{\i}var, Caracas, Venezuela}


\begin{abstract}
This work concerns with the following problem.
Given a two-dimensional domain whose boundary is a closed polygonal line with 
internal boundaries defined also by polygonal lines,
it is required to generate a grid consisting only of quadrilaterals with the 
following features: 
(1) conformal, that is, to be a tessellation of the two-dimensional domain such
that the intersection of any two quadrilaterals is a vertex, an edge or empty 
(never a portion of one edge),
(2) structured, which means that only four quadrilaterals meet at a single node
and the quadrilaterals that make up the grid need not to be rectangular, and 
(3) the mesh generated must be supported on the internal boundaries.
The fundamental technique for generating such grids, is the deformation of 
an initial Cartesian grid and the subsequent alignment with the internal boundaries.
This is accomplished through the numerical solution of an elliptic partial 
differential equation based on finite differences.
The large nonlinear system of equation arising from this formulation is solved through
spectral gradient techniques.
Examples of typical structures corresponding to a two-dimensional, areal
hydrocarbon reservoir are presented.
\end{abstract}


\end{frontmatter}

\nolinenumbers

\section{Introduction} \label{QGG-intro}

Nowadays the crude oil is not only one of the world's most important combustibles 
but also it becomes almost half of the energy consumption in the world. 
For these reasons, it is important to take care and to correctly administrate the earth's 
existing oil reserves. Also, the ability to predict the performance of a petroleum 
reservoir is of immense importance for the petroleum industry.

A petroleum reservoir is the place where oil stores naturally. Every reservoir is 
unique based on its geological or geophysical characteristics, and the type of crude 
oil the formation contains. Wells are drilled into oil reservoirs to extract the 
crude oil. All this makes the oil extraction and exploitation methods to be very costly.

For obvious reasons it is very useful to predict the behavior of the reservoir before 
and during its exploitation. One would like to be able to know as much as possible 
about production rates and total production resulting from different production strategies.

The petroleum engineer needs to understand at least, the complex reservoir structure and 
the fluid movement through it, to generate an exploitation plan in order to make the right 
predictions of the reservoir production.

At this point, it is where the numerical reservoir simulators play an important role in the reservoir 
exploitation. To this end, numerical reservoir simulation has gained wide acceptance 
as an important decision-making tool, and has become the industry standard for reservoir 
management.

By numerical reservoir simulation we mean the process of inferring the behavior of 
a real reservoir from the performance of a mathematical model of that physical system.
Traditionally, the petroleum engineer represents mathematically an oil reservoir through 
its discrete model using rectangular meshes, in order to perform numerical reservoir 
simulations to predict the oil production in time.
However, the geological structure of the oil reservoir is not adequately represented
using rectangular meshes.
The discrete model of the reservoir must honor the natural boundaries of the reservoir
(see \citep{Evasi_etal2009}).
Given today's reality, the petroleum engineer should align the grid, for instance, 
with the principal directions of deposition and to the preferential flow directions.
A grid (the discrete model) that adequately represents a hydrocarbon reservoir is 
fundamental to be used with any mathematical model that allows to calculate the total 
amount of oil that will ever be recovered according to probable scenarios. 
For example, when or where to drill additional producer wells, injector wells, 
when to shut in wells, to inject water or gas, or to prove other techniques for increasing 
the amount of crude oil that can be extracted from an oil field.

For our purposes, the mathematical model is a set of partial differential equations 
with an appropriate set of boundary conditions, which describes the significant physical 
processes taking place in the system (the reservoir).

The mathematical modeling of the multi-phase fluid flow in a porous media then requires 
grids that represent the complexity of the reservoir structure. On this relay the 
interest of generating certain types of grids that honor the reservoir structure 
through their alignment to a set of internal boundaries.
There are two types of grids, structured or unstructured, that can solve this problem.
A structured mesh can be recognized by all interior nodes of the mesh having an equal 
number of incidence elements. Unstructured mesh generator, on the other hand, relaxes 
the node valence requirement, allowing any number of elements to meet at a single node.

A lot of studies concerning grid generation have been conducted under structured/unstructured 
and/or orthogonal/non-orthogonal grids using finite elements and finite differences. 
Various shaped meshes are widely introduced, such as tetrahedral or hexahedral meshes 
under non-orthogonal or unstructured mesh system, which are very suitable for fitting 
complex geometry.

This work will be confined to structured meshes in two-dimensional domains corresponding 
to areal (2D) or transversal views of oil reservoirs. From this type of grids, 
a three dimensional mesh can be generated: starting from a two dimensional grid on a
reference plane, sweep it through space along a curve between a source and 
target surface.
The source and target surface correspond to layers of the 3D domain 
(the reservoir). It is sometimes referred to as 
2{\small $\frac{1}{2}$}D meshing.

The rectangular meshes are inappropriate to model any type of internal boundaries. 
For example, in case of internal boundaries forming different angles with the 
coordinate axes, a rectangular mesh will give stairs shaped paths to model the boundaries,
which could distort the flow field in the vicinity of that boundary.
That is why the interest of this work will be concentrated on quadrilateral meshes.

In summary, given a rectangular domain, the reservoir, this work is concerned with the 
generation of a 2D non orthogonal structured grid consisting only of quadrilaterals that honor 
the internal boundaries of the domain.

Some publications related to the goal of the present work will be detailed next.

Winslow (1966) in \citep{Winslow1997} proposed a method
to generate a triangular mesh by solving Laplace’s equation.
The method can be derived by formulating the zoning problem as a potential problem, 
with the mesh lines playing the role of equipotentials.
Because of the well-known averaging property of solutions to system of Laplace equations,
we might expect a mesh constructed in this way to be, in some sense, smooth.
Winslow said this methodology is easily adapted to non rectangular boundaries and interfaces.

Amsden and Hirt (1973) in \citep{AmsdenHirt1973}
proposed an iterative process to transform a rectangular grid 
into a more complex configuration, by the direct solution of a Laplace equation type. 
The final grid is made of quadrilateral elements only. Even though the domain does 
not have internal boundaries, the methodology has the potential to handle such type 
of constrains.

Thompson et al. (1974) in \citep{Thompson_etal1974,Thompson_etal1977} extended the works in 
\citep{Winslow1997} and \citep{AmsdenHirt1973} to generate quadrilateral meshes 
as solutions of an elliptic differential system to multiconnected 
regions with any number of arbitrary shape bodies or holes.
Their work was confined to two dimensions in the interest of compute economy,
but all techniques are immediately extendable to three dimensions.

Knupp (1992) in \citep{knupp1992} proposed a variational principle that results in a 
robust elliptic grid generator having many of the strengths of \citep{Winslow1997}
and \citep{AmsdenHirt1973,Thompson_etal1977}. This grid generator places grid 
lines more uniformly over the domain without loss of orthogonality.
Grid quality measures were introduced to quantify differences between discrete grids.
The author states that generalization of the new method to surface and volume
grid generation is straightforward.

Borouchaki and Frey (1996 and 1998) in \citep{Borouchakietal1996,BorouchakiFrey1998}
proposed a method to generate quasi quadrilateral meshes based 
on an automatic triangular to quadrilateral mesh conversion scheme. This method allows 
refining the mesh in order to capture the behavior of the underlying physical phenomenon,
but the resulting mesh is not entirely composed by quadrilaterals.
The main contribution is to extend the triangle merging procedure to the case
where a generalized metric map is specified.
They also introduced a new mesh optimization technique based on vertex smoothing.

Sarrate and Huerta (2000) in \citep{SarrateHuerta2000}
described an algorithm for automatic unstructured quadrilateral 
mesh generation based on a recursive decomposition of the domain into quadrilateral 
elements. Two facts for the generated mesh can be highlighted: (1)
there is no need for a previous step where triangles are generated, and (2)
the generated quadrilateral mesh may have more or less than four elements
meeting at a single node.
The model is applied to complex geometry domains without internal boundaries.

Hyman et al. (2000) in \citep{Hyman2000}
presented a numerical algorithm that aligns a structured quadrilateral grid with 
internal alignment curves. These curves represent internal boundaries that can be 
used to delineate internal interfaces, discontinuities in material properties, 
internal boundaries, or major features of a flow field. The authors use
Gauss-Seidel iterations to solve the Thompson, Thames and Mastin (TTM) smoothing equations
(see \citep{Thompson_etal1974,Thompson_etal1977}).
The smoothing regularizes the distribution of the grid points, is guaranteed to converge,
and eliminates overlapping grid cells.
The Gauss-Seidel iterations are halted when the residual of the TTM equations falls below
a small tolerance.
For complex internal boundaries the authors admit that is extremely difficult to automate 
the procedure.

Sarrate and Huerta (2002) in \citep{SarrateHuerta2002}
presented the extension of the unstructurated and 
quadrilateral grid generation algorithm in \citep{SarrateHuerta2000} 
to three dimensional parametric surfaces.
The target of this extension is to build the discretization
in the plane of parameters and then map the obtained mesh on the surface
according to its geometric properties.

Kyu-Yeul et al. (2003) in \citep{Kyu-Yeul_etal2003}
proposed an algorithm, based on the constrained Delaunay triangulation and 
Q-Morph algorithm, to automatically generate a 2D unstructured quadrilateral mesh 
which handles line constraints. The triangulation method is the one who handle line constraints.
Q-Morph utilizes an advancing front approach to combine triangles into quadrilaterals. 
Kyu-Yeul et al. implemented the constrained Laplacian smoothing method to improve mesh 
quality. This methodology does not involve the resolution of any nonlinear system.


Lin et al. (2007) in \citep{Lin_etal2007}
presented a B\'ezier patch mapping algorithm, based on a bijective boundary-conforming
mapping method, which generates a strictly non-self-overlapping structured  
quadrilateral grid in a given four-sized planar region, whose boundaries are polynomial curves.
Finally, a constrained optimization problem is formualted in order to ensure the bijectiveness of
B\'ezier patch mapping. This problem is solved using the Matlab optimization library.
This methodology does not handle any type of domain internal boundaries.

Parka et al. (2007) in \citep{Parka_etal2007}
developed an automated method of quadrilateral mesh with random line constraints. The algorithm
is based on advanced front techniques and a direct method to handle line-typed features automatically
without any user interactions and modification. The generated mesh does not require a previous
triangular mesh, and in general it is not structured.

Villamizar et al. (2007 and 2009) in \citep{Villamizar_etal2007,Villamizar_etal2009}
proposed a 2D elliptic grid generator to create a structured quadrilateral smooth mesh with 
boundary conforming coordinates and grid lines control, 
on multiply connected regions including boundary singularities.
The generator is based on numerical solution of a Poisson equations system
and the grid line spacing is controlled by a specific nodal distribution on appropriated 
boundary curves. This technique does not take into account any type of internal boundaries.

Berndt et al. (2008) in \citep{Berndt_etal2008} presented two Jacobian-Free Newtown-Krylov (JFNK) 
solvers for Laplace-Beltrami grid generation system of equations.
The two JFNK solvers differ only in the preconditioner.
A key feature of these methods is that the Jacobian is not formed explicitly.

Ruiz-Giron\'es and Sarrate (2008 and 2010) in \citep{RuizGironesetal2008,RuizGironesetal2010} 
proposed a modification of the submapping method to generate structured quadrilateral meshes, 
in order to be applied to geometries in which the angle between two consecutive 
edges of its boundary is not an integer multiple of $\pi/2$.
The submapping method splits the geometry into pieces logically equivalent to a quadrilateral,
and then, meshes each piece keeping the mesh compatibility between them by solving an integer 
linear problem. They used the transfinite interpolation method (TFI) to mesh each patch.
In addition, the authors proposed a procedure to apply it to multiply connected domains.
Also, they present several numerical examples that show the applicability of the 
developed algorithms.

Khattri (2009) in \citep{Khattri2009} presented the elliptic grid generation 
system for generating adaptive quadrilateral
meshes, and its implementation in the C++ language.
The coupled elliptic system are linearised by the method of finite differences, 
and the resulting system is solved by the SOR relaxation.
The presented method generates adaptive meshes without destroying the structured
nature of the mesh.
This technique does not take into account any type of internal boundaries.

Evasi-Yadecuri and Mahani (2009) in \citep{Evasi_etal2009} presented
a novel unstructured (coarse) grid generation approach using structured background grid.
A structured/cartesian fine grid distribution of properties, either static or dynamic, 
is utilized to create background grid or spacing parameter map.
Once background grid is generated, advancing front triangulation and 
then Delaunay tessellation are invoked to form the final (coarse) gridblocks.
This technique does not take into account any type of internal boundaries.

Liu et al. (2011) in \citep{liuetal2011} presented an indirect approach for automatic 
generation of unstructured quadrilateral mesh with arbitrary line constraints.
The methodology follows the steps: (1) discretizing the constrained lines within 
the domain; (2) converting the above domain to a triangular mesh together with the 
line constraints; (3) transforming the generated triangular mesh with line constraints 
to an all-quad mesh through performing an advancing front algorithm from the line 
constraints, which enables the  construction of quadrilaterals layer by layer, and
roughly keeps the feature of the initial triangular mesh; (4) optimizing the topology 
of the quadrilateral mesh to reduce the number of irregular nodes; (5) smoothing the
generated mesh toward high-quality all-quad mesh generation.

Rathod et al. (2014) in \citep{rathodetal2014} described a scheme for unstructured 
quadrilateral mesh generation of a convex, non-convex polygon and multiple 
connected linear polygon. They decompose these polygons into simple sub regions 
in the shape of triangles. These simple regions are then triangulated to generate 
a fine mesh of triangular elements. Finaly, the authors proposed an automatic 
triangular to quadrilateral conversion scheme. Although the paper describes 
the scheme as applied to planar domains, the author states that it could be 
extended to 3D.

Fortunato et al. (2016) in \citep{fortunatoetal2016} proposed the generation of 
unstructured high-order meshes by solving the classical Winslow equations.
They described a new continuous Galerkin finite element formulation 
of the standard Winslow equations, which they use for generation of well-shaped high-order 
unstructured curved meshes. Compared to other finite element formulations in the literature, 
their discretization attempts to directly mimic the non-conservative form used by most 
finite difference solvers, which allows for a highly efficient Picard solver.

The methodology proposed in the present work
follows the technique proposed by Hyman et al. (2000) in \citep{Hyman2000},
also see \citep{Borregales_etal2009} and \citep{Valido_etal2012},
to generate grids that represent the complexity of the reservoir structure,
because it does not require a previous triangular mesh, the resulting quadrilateral mesh is well
adapted to internal boundaries and is structured and non necessarily orthogonal. 
The generation of such grid is accomplished through the numerical solution of 
a system of elliptic partial differential equation (see \citep{Thompson_etal1985} and 
\citep{Khattri2009}) based on finite differences.
The large nonlinear system arising from this discretization are solved using spectral 
gradient techniques (see \citep{LaCruzRaydan2003}), which are low in storage and 
low in computational cost.
Buitrago et al. (2015) in \citep{Buitragoetal2015}, proposed a numerical model 
based on finite volume methods for the solution of the 2D convection 
diffusion equation on the type of meshes presented in this work, i.e.
non-rectangular grids formed only by quadrilaterals
honouring the internal structures of the domain.

\section{Formulation of the problem} \label{QGG-formulacion}

This work concerns with the following problem.

Given a two-dimensional domain $\Omega$ whose external boundary is a closed polygonal line with 
internal boundaries defined also by polygonal lines, it is required to generate a grid 
consisting only of quadrilaterals with the following features:

\begin{enumerate}
\item conformal, that is, to be a tessellation of the two-dimensional domain $\Omega$ 
such that the intersection of any two quadrilaterals is a vertex, an edge or empty 
(never a portion of one edge),
\item structured, but non cartesian, this means that the quadrilaterals that 
make up the grid need not to be rectangular and only four quadrilaterals meet 
at a single node, and
\item the mesh generated must be supported on the internal boundaries.
\end{enumerate}

It is also considered the possibility that interior points in $\Omega$ be vertices 
of the resulting grid. These points will represent wells on the reservoir to simulate. 
Internal boundaries, that may have a complex layout and configuration, mimic the 
structure of the reservoir.

The fundamental technique for generating such type of grids, is the deformation 
of an initial Cartesian grid and the subsequent alignment with the internal boundaries. 
This is accomplished through the numerical solution of an elliptic partial differential 
equation based on finite differences.

The internal boundaries must be the representation of the geometric structures 
of the reservoir. Examples of these are boundaries between soil types, preferential 
fluid channels, zones with different permeability, system of fractures, etc.

\bigskip
\noindent
{\bf Modelling the internal boundaries.}

Let's define $\Omega = [a,b] \times [c,d] \subset \mathbb{R}^{2}$.
Internal boundaries of $\Omega$ will be modeled using four types of polygonal lines.

\begin{Definition} \label{IAC}
Let's define an IAC (Internal Alignment Curves) as a polygonal line contructed 
with the concatenation of line segments that do not touch the external boundary 
of $\Omega$ and do not intersect each other.
\end{Definition}

\begin{Definition} \label{SIAC}
Let's define a SIAC (Spanned IAC) as a polygonal line generated by the concatenation of
line segments, such that this goes horizontal or vertical from one side of $\Omega$ 
to the other, but never from a horizontal side to a vertical one or vice
versa. In the last case, it is said that the SIAC is mal formed. The SIAC
breakpoints and the ends, as well as the possible intersection points between
SIAC, will be called vertices.
\end{Definition}

The IAC are extended to the external boundaries to form horizontal or vertical SIAC.

Given a horizontal or a vertical SIAC, going from one end to the other, either 
from left to right or from bottom to top, the vertices are enumerated to produce 
a sequence $v_1, \dots, v_k$. 

\begin{Definition} \label{SIACgrowthproperty}
It is said that a SIAC is growing if for each 
pair of vertices $v_i$ and $v_j$ holds
\begin{itemize}
\item If $i$ is less than $j$ and it is a horizontal SIAC, then the abscissa 
of $v_i$ is less than the abscissa of $v_j$.
\item If $i$ is less than $j$ and it is a vertical SIAC, then the ordinate 
of $v_i$ is less than the ordinate of $v_j$.
\end{itemize}
\end{Definition}

In this work, all growing SIAC will be grouped in horizontal or vertical SIAC.
From now on, when we mention a SIAC it will refer to a growing SIAC.

\begin{Remark}
In this work we assume that all growing SIAC satisfy the following features:
\begin{itemize}
\item they should always be growing in some direction, either horizontally or 
vertically,
\item any two horizontal SIAC never intersect each other, 
 the same holds for any two vertical SIAC, and
\item a horizontal SIAC intersects with a vertical one at some point in $\Omega$, 
 and only one.
\end{itemize}
\end{Remark}

\begin{Definition} \label{QIAC}
Let's define a QIAC (Quadrilateral IAC) as a convex closed four sided 
polygonal line that does not touch the
external boundary of $\Omega$. Each QIAC has a minimal rectangle; it is the
smallest rectangle with edges parallel to the coordinate axes, which contains the
QIAC. Similarly, a subdomain of a QIAC is any rectangle, with edges parallel
to the coordinate axes, which properly contains the QIAC minimal rectangle.
\end{Definition}

Fig.\ref{internal_boundaries} presents different types of polygonal lines (SIAC and QIAC).

\begin{Definition} \label{associatedQIAC}
Given two QIAC, it is said that they are associated either when the minimal 
rectangles to each QIAC intersect in at least a point, or when 
the minimal rectangle of one QIAC intersects the other QIAC 
(see an example of associated QIAC in Fig.\ref{genIQIAC}). 
\end{Definition}

\begin{Definition} \label{IQIAC}
Let's define an IQIAC as a  set of associated QIAC, 
that have been concatenated through artificial vertices and line segments, 
till the new structure constitutes a conformal partition of convex quadrilaterals
(see an example in Fig.\ref{genIQIAC}).
\end{Definition}

The minimal rectangle and subdomain definitions given for a QIAC could be 
extended similarly to the case of an IQIAC.

\begin{Remark} \label{QIAC-2vertices}
QIAC and IQIAC must have the following properties:
\begin{itemize}
\item A QIAC has two vertices on the horizontal lines and two 
vertices on the vertical lines. 
\item For an IQIAC, its horizontal lines have the same amount of vertices, 
and the same thing is true for all its vertical lines
(see the IQIAC in Fig.\ref{genIQIAC}).
\end{itemize}
\end{Remark}


All vertical and horizontal lines of either a QIAC or an IQIAC will be 
extended to the external boundaries to form
SIAC, in a similar way as an IAC.

\begin{Lemma} \label{QIAC-SIAC-vertices}
Given a QIAC or an IQIAC, the amount of vertical SIAC generated 
coincides with the amount of vertices on the horizontal 
lines of the QIAC or IQIAC. 
It occurs similarly with the horizontal SIAC generated.
For example see the image on the left in Fig.\ref{Fig1}).
\end{Lemma}

\begin{proof}
it follows straightforward from definitions \ref{QIAC} of QIAC and \ref{IQIAC} of IQIAC, 
and Remark~\ref{QIAC-2vertices}.
\end{proof}

\begin{figure}[ht]
\noindent \begin{centering}
{\includegraphics[width=2.1in,height=2.1in]{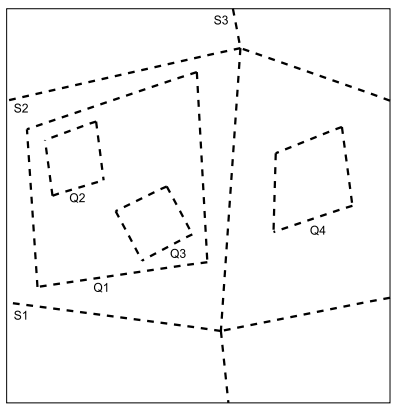}}
\par\end{centering}

\noindent \centering{}\caption{Some examples of internal boundaries:
SIAC (S1, S2 and S3) and QIAC (Q1, Q2, Q3 and Q4). The interior of QIAC Q1 
contains the QIAC Q2 and Q3.}
\label{internal_boundaries}
\end{figure}

\begin{figure}[ht]
\noindent \begin{centering}
{\includegraphics[width=3in,height=1.4in]{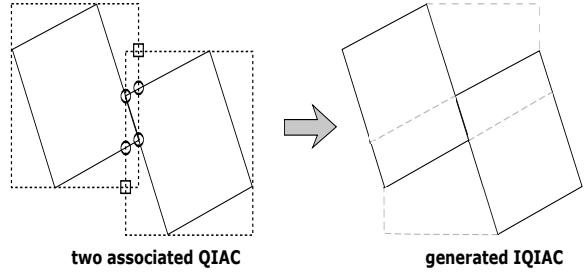}}
\par\end{centering}

\noindent \centering{}\caption{On the left: two associated QIAC, dotted rectangles
are the QIAC minimal rectangles, points marked with squares are minimal 
rectangles intersections, and points marked with circles are intersections of minimal rectangles
and internal boundaries. On the right: generated IQIAC.}
\label{genIQIAC}
\end{figure}

The fact that QIAC and IQIAC could be thought as a set of SIAC, allows 
using a unified methodology for the treatment of SIAC. This is not the 
case for original SIAC or those generated from IAC.

\section{Methodology proposed for the generation of the 
  structured quadrilateral mesh} \label{QGG-metodologia}

\noindent
The proposed general process for generating the desired mesh, 
follows these steps:
\begin{enumerate}
\item Generating an initial cartesian grid
\item Processing the SIAC generated by the QIAC
\item Processing the SIAC generated by the IQIAC
\item Processing the IAC
\item Either processing the original SIAC or applying a global 
smoothing in case there are no original SIAC
\end{enumerate}

Each of these steps is detailed below.

\bigskip
\noindent
{\bf Initial cartesian grid.}

The process starts with the generation of a standard cartesian grid with constant spacing
in each direction. For each direction the spacing is a fraction of the minimal distance between 
the intersection points of the SIAC with the external boundaries, as shown, for example, in 
Fig.\ref{Fig1}.
We denote by $m$ and $n$ the numbers of partition points in each direction. Thus, the initial
cartesian grid has $mn$ nodes.

\begin{figure}[ht]
\noindent \begin{centering}
{\includegraphics[width=2.1in,height=2.1in]{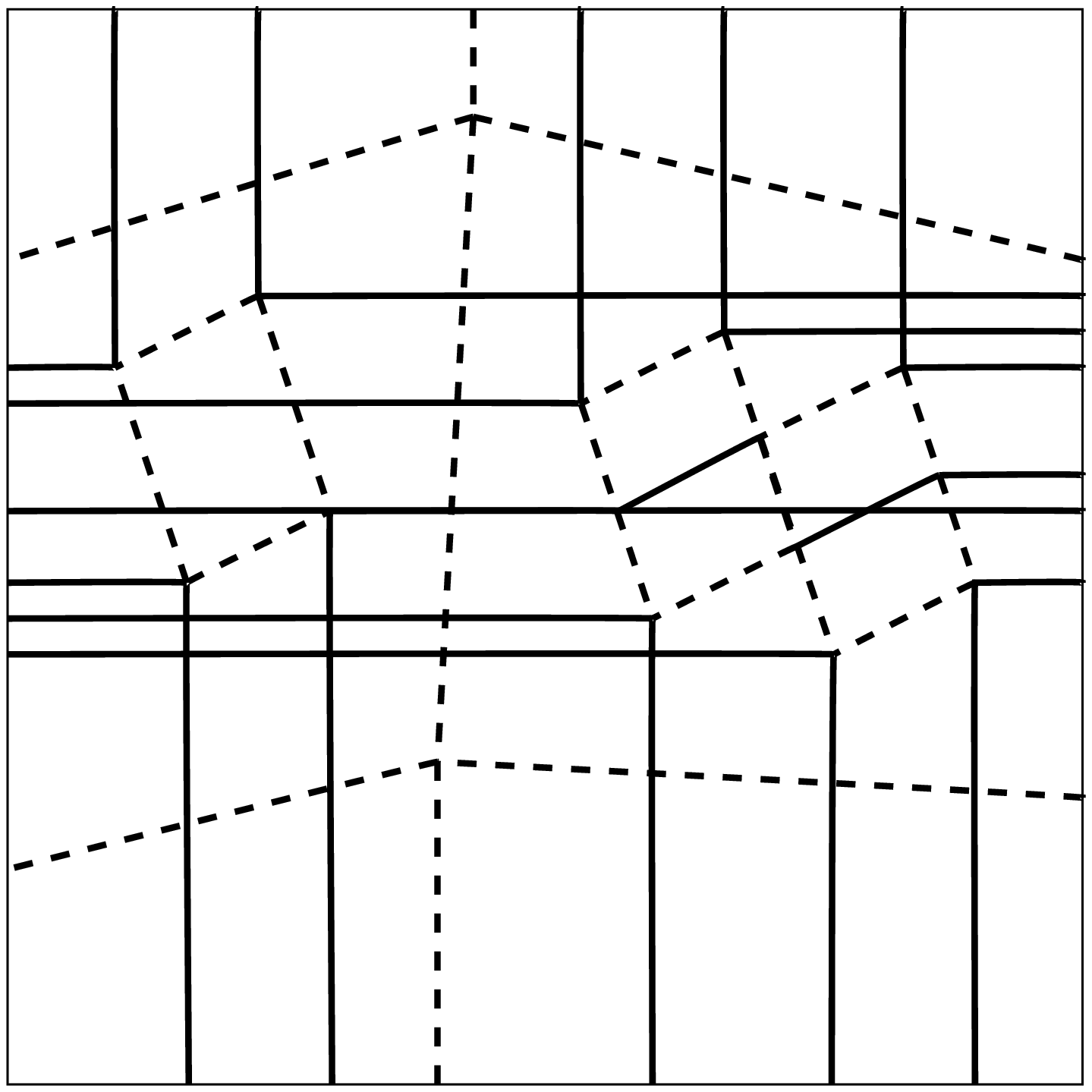}
\includegraphics[width=2.1in,height=2.1in]{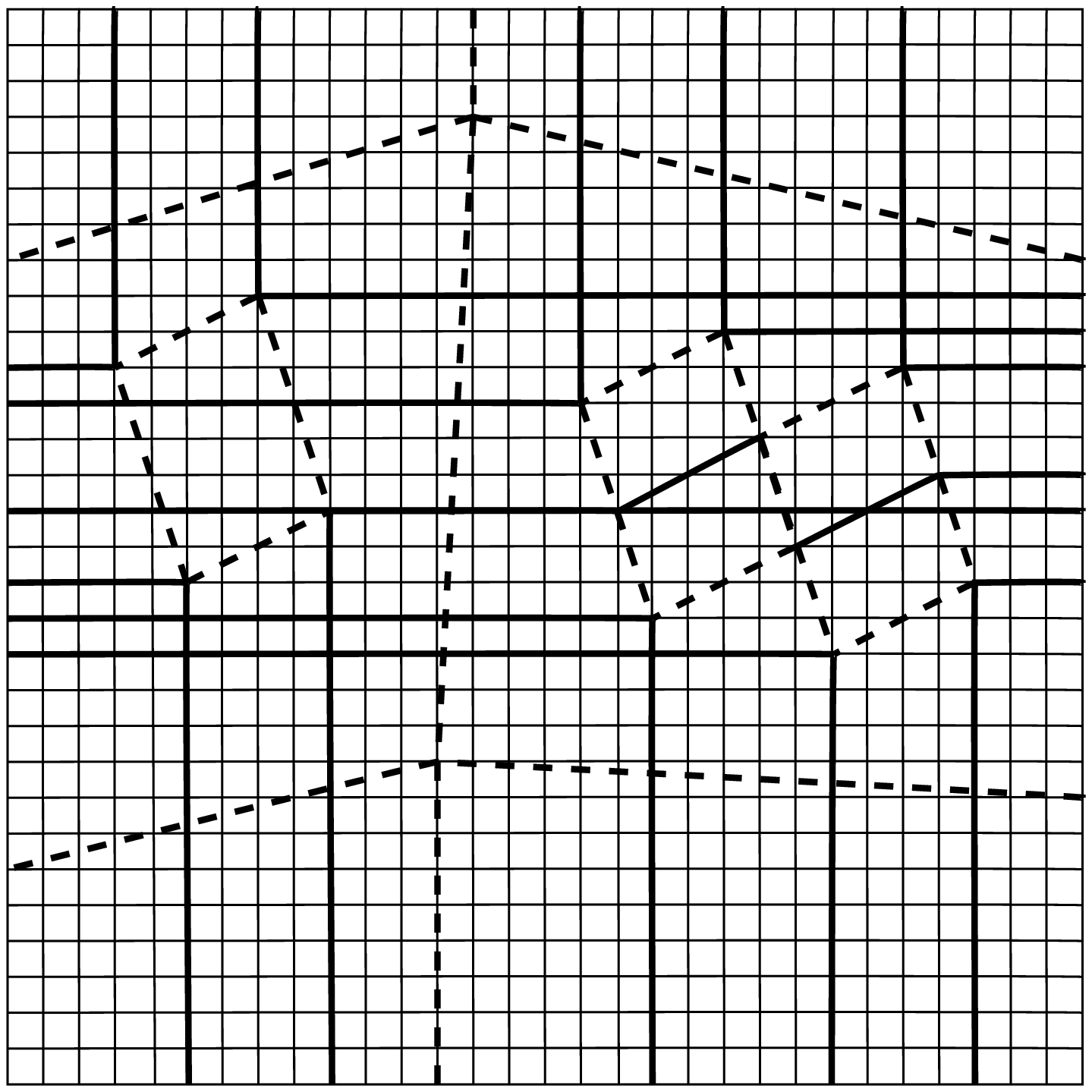}}
\par\end{centering}

\noindent \centering{}\caption{On the left: internal boundaries identified with dotted lines,
three SIAC, one QIAC and one IQIAC. All vertical and horizontal lines of either a QIAC or 
an IQIAC are extended to the external boundaries to form SIAC.
On the right: Example of the initial cartesian grid based on the internal boundaries.}
\label{Fig1}
\end{figure}

\bigskip
\noindent
{\bf Processing the SIAC and IAC.}

As it was explained in the last section, the QIAC and IQIAC are decomposed into a set of SIAC,
while the IAC are extended to the external boundaries to form SIAC. Consequently, the processing of
all these internal boundaries is reduced to the treatment of a set of SIAC. The aim of the SIAC 
treatment is to adjust the nodes of the initial cartesian grid to this set of SIAC. 

This procedure consists of three essential steps:
\begin{itemize}
\item association of lines,
\item redistribution of nodes on associated lines, and
\item labelling of fixed nodes
\end{itemize}

\medskip
\noindent
{\it Association of lines.}

Given, for example, a horizontal SIAC, the idea is to associate a line in the cartesian grid that 
best adjust to that SIAC.
To achive this, calculate the mean of the vertices ordinates of the SIAC and then choose 
a line of the cartesian grid whose ordinate is the closest to that mean.
In the case of a vertical SIAC, proceeds in similar way. See the top-left image in Fig.\ref{Fig2}.

\begin{figure}[ht]
\begin{centering}
{\includegraphics[width=2.2in,height=2.2in]{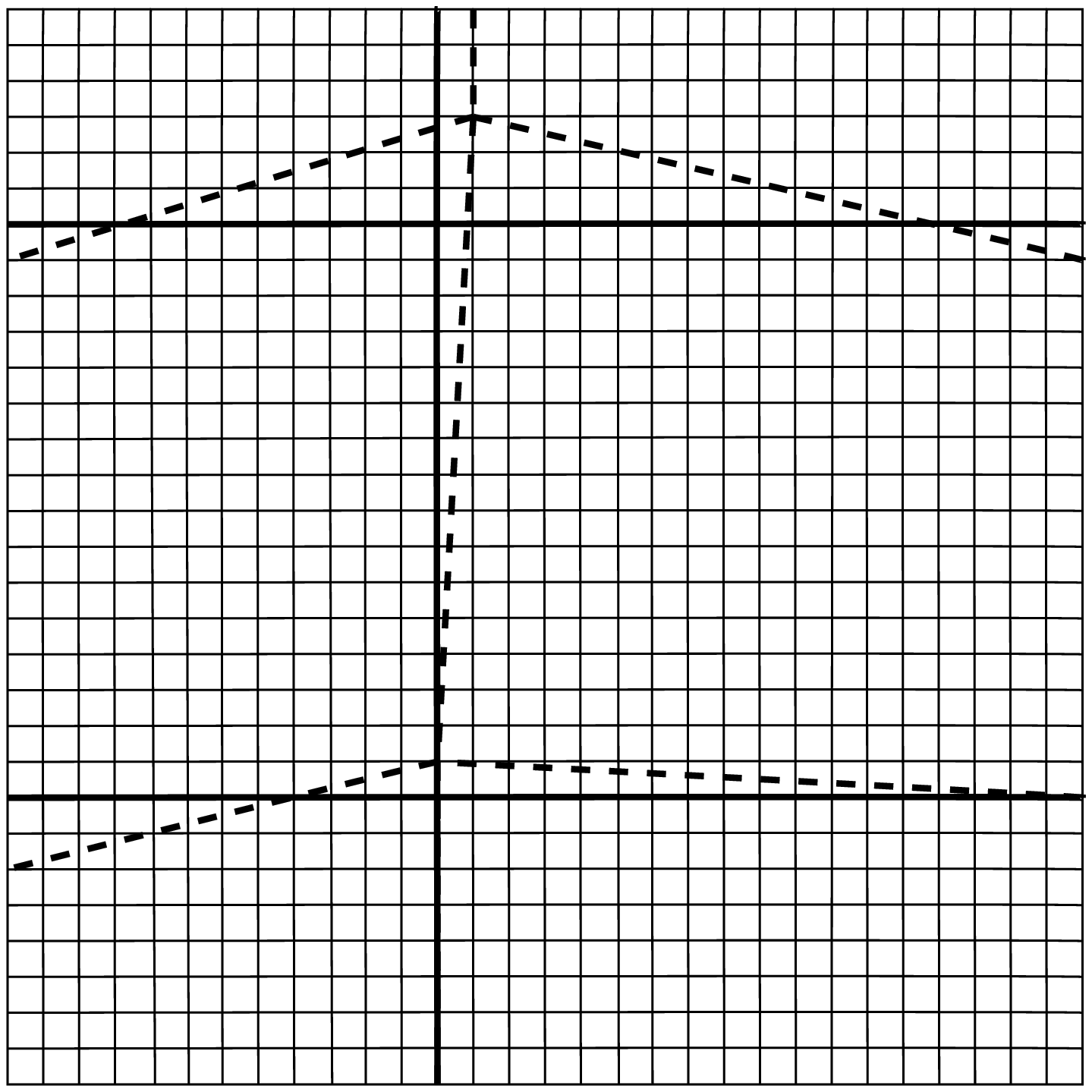}
\includegraphics[width=2.2in,height=2.2in]{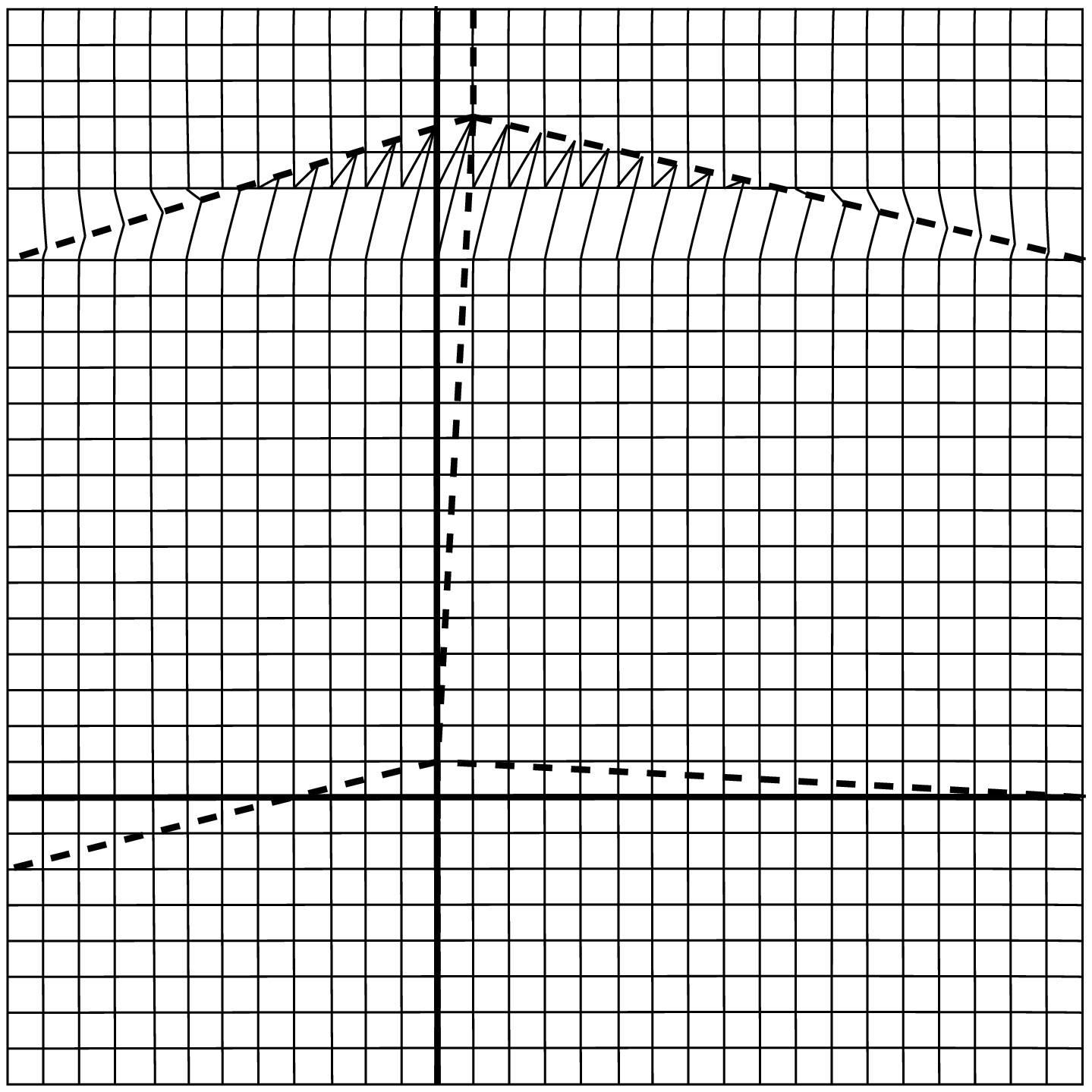}
\includegraphics[width=2.2in,height=2.2in]{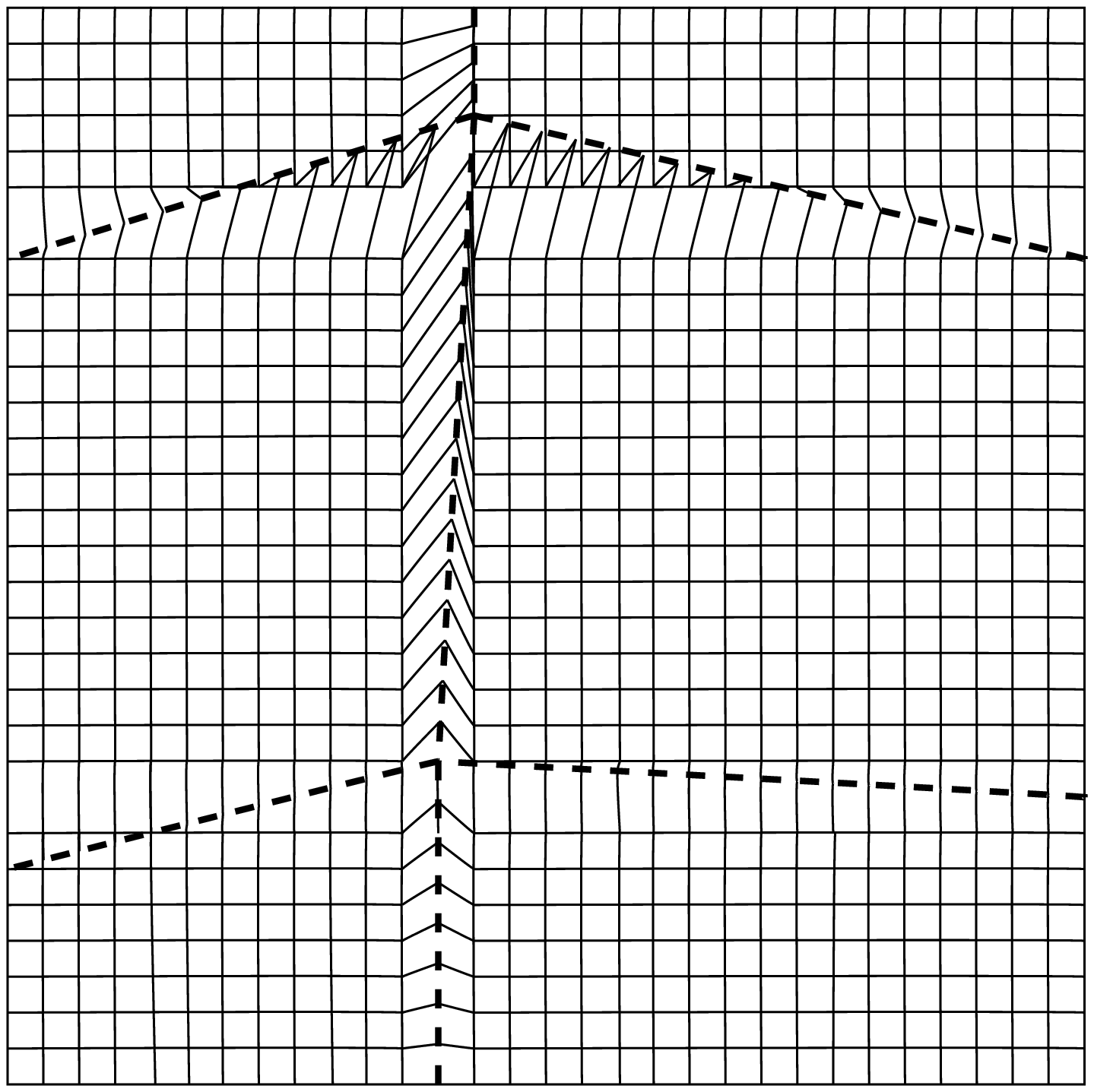}
\includegraphics[width=2.2in,height=2.2in]{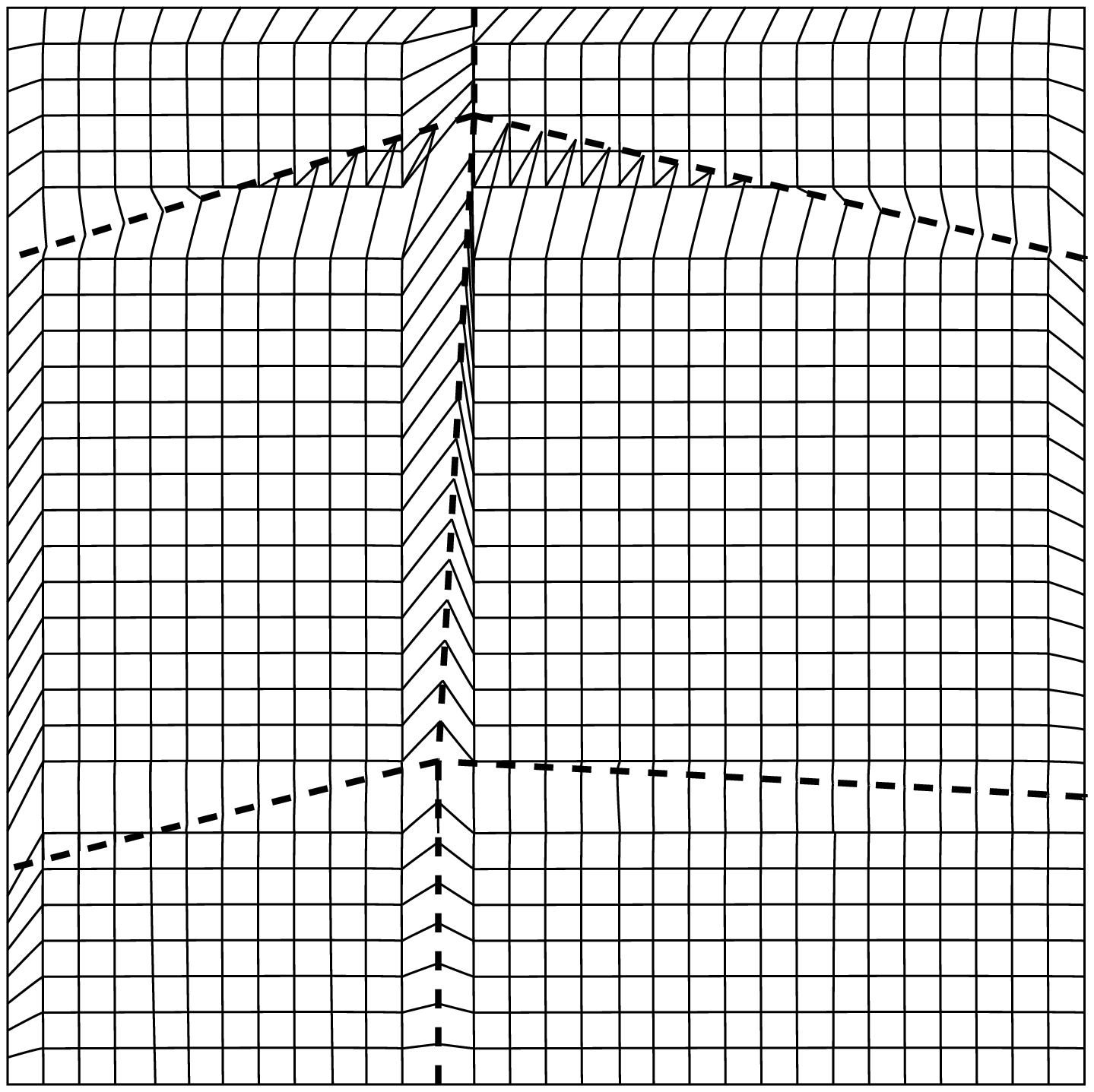}}
\par\end{centering}
\caption{Gray lines conform the initial catersian grid, dotted lines are internal boundaries (SIAC). 
 Top-left: thick lines of the mesh represent associated lines to each SIAC.
 Top-right: mapping the mesh nodes of the upper horizontal associated line on the SIAC.
 Bottom-left: overlap mesh generated from the three SIAC.
 Bottom-right: overlap mesh generated from the three SIAC and the external boundaries.}
\label{Fig2}
\end{figure}

\medskip
\noindent
{\it Redistribution of nodes on the associated line.}

\noindent
First, the end vertices of the SIAC are associated to the end points of the associated line.
Second, each internal vertex of the SIAC will be assigned to the closest internal node of the 
associated line. Therefore, a partition is created on each associated line, 
based on the vertices of the SIAC.
Finally, for each SIAC and associated line, one proceeds as follows:
the nodes for each pair of consecutive internal vertices on the associated line, are
distributed in a proportional way over the corresponding segment of the SIAC, 
see top-right and bottom-left images in Fig.\ref{Fig2}.

It is important to point out that a vertex corresponding to the intersection of two SIAC, 
is always associated to the intersection of the two associated lines.

A similar redistribution procedure applies to the external boundary nodes.

Therefore an overlap mesh, coming from the deformation of the initial cartesian grid, is obtained
(see bottom-right image in Fig.\ref{Fig2}).

\medskip
\noindent
{\it Labelling of fixed nodes.}

Some of the nodes of the overlap grid must be labeled, because they will be fixed during 
the smoothing process. The fixed nodes are either the ones coming from the
original internal boundaries (IAC, QIAC, IQIAC and original SIAC) or those belonging to the 
external boundary. This labelling information will play an important role in the resolution of
the non linear system arising from the discretization of the elliptic partial differential
equation in the smoothing process.

\bigskip
\noindent
{\bf Smoothing process.}

This process seeks to relocate in a smooth manner the nodes belonging to the overlapping 
mesh in order to get the quadrilateral mesh, structured, conformal and tailored to the 
internal boundaries of $\Omega$.

The transformation of the physical plane $(x,y)$ (overlapping mesh) to the 
transformed plane $(\xi,\eta)$ (smooth mesh) is given by the vector function 
$f(x,y)=(\xi,\eta)$, with $\xi=\xi(x,y)$, $\eta=\eta(x,y)$.
Similarly, the inverse transformation is given by the functions
$x=x(\xi,\eta)$, $y =y(\xi,\eta)$. The Jacobian matrix of the inverse transformation
is given by
$$\frac{\partial(x,y)}{\partial(\xi,\eta)} =
\begin{pmatrix} x_{\xi}   & x_{\eta}  \\ y_{\xi}  & y_{\eta} \end{pmatrix},$$

\noindent
and $J = \det \frac{\partial(x,y)}{\partial(\xi,\eta)} = x_{\xi} y_{\eta} - x_{\eta} y_{\xi}$,
where it is imposed that $x_{\xi} y_{\eta} - x_{\eta} y_{\xi} \ne 0.$

The smooth mesh satisfies, see 
\citep{Winslow1997,Thompson_etal1974,Thompson_etal1977,Thompson_etal1985},
the standard system of elliptic equations
\begin{equation}
\begin{array}{l}
\label{eq:EllipticEq}
\Delta \xi = \xi_{xx} + \xi_{yy} = 0, \\
\Delta \eta = \eta_{xx} + \eta_{yy} = 0,
\end{array}
\end{equation}

\noindent
with Dirichlet type boundary conditions, corresponding to the positions of fixed nodes on the 
internal and external boundaries.

The tangent vectors at a point $P$ are defined by
$$e_1 = (\frac{\partial x}{\partial \xi}, \frac{\partial y}{\partial \xi})^t \ \ {\rm and} \ \
e_2 = (\frac{\partial x}{\partial \eta}, \frac{\partial y}{\partial \eta})^t \ .$$
These base vectors are called covariant base vectors.

A second set of vectors, $e^1$ and $e^2$, is defined by 
$e^i \cdot e_j = \delta_j^i$, $i,j=1,2$, with $\delta$ the Kronecker delta.
The $e^i$ are called contravariant base vectors and are orthogonal to
the respective covariant vector for $i \ne j$.

A vector $v$ can either be represented in the covariant or contravariant base vectors as
$$v = \sum_i v^i e_i = \sum_j v_j e^j \ ,$$
\noindent
where $v^i$ and $v_j$ are their components respectively.

Let's define the matrix ($g_{ij}$) as $g_{ij} = e_i \cdot e_j$ for $i,j=1,2$, 
then $g = \det (g_{ij}) = g_{11} g_{22} - g_{12}^2$.
It follows,
\begin{equation}
\label{eq:matrix_g}
\begin{array}{l}
g_{11} = (x_\xi)^2 + (y_\xi)^2, \ \ g_{12} = g_{21} = x_\xi x_\eta + y_\xi y_\eta,
\ \ g_{22} = (x_\eta)^2 + (y_\eta)^2 \ , \\
{\rm and} \ \ \ g = (x_{\xi} y_{\eta} - x_{\eta} y_{\xi})^2 = J^2 \ .
\end{array}
\end{equation}

The components of the inverse matrix ($g^{ik}$) of $(g_{ik})$ are found from 
$$\sum_{i=1}^2 g_{ij} g^{ik} = \delta_j^k \ .$$
\noindent
Therefore, after some calculations,
\begin{equation}
\label{eq:inv-matrix_g}
g^{11} = g^{-1} g_{22}, \ \ \ g^{12} = g^{21} = -g^{-1} g_{21}, \ \ \
g^{22} = g^{-1} g_{11} \ .
\end{equation}

Some of the lemmas and theorems~\ref{EllipticInvertedEq} and~\ref{DiscreteEllipticInvertedEq} 
included in this subsection are for the sake of completeness 
(see~\citep{Knupp_etal1993}).

\begin{Lemma} \label{JacMatrix_f}  
The Jacobian matrix of the transformation $f$ is
\begin{equation}
\frac{\partial(\xi,\eta)}{\partial(x,y)} =
\begin{pmatrix} \xi_{x}   & \xi_{y}  \\ \eta_{x}  & \eta_{y} \end{pmatrix} =
J^{-1} \begin{pmatrix} y_{\eta}   & -x_{\eta}  \\ -y_{\xi}  & x_{\xi} \end{pmatrix},
\end{equation}
with $J = x_{\xi} y_{\eta} - y_{\xi} x_{\eta}$.
\end{Lemma}

\begin{proof}
let $a$, $b$, $c$ and $d$ be the components of the inverse matrix of
$\frac{\partial (x,y)}{\partial (\xi,\eta)}$, i.e.
$$\begin{pmatrix} a & b \\ c & d \end{pmatrix} 
\begin{pmatrix} x_{\xi}   & x_{\eta}  \\ y_{\xi}  & y_{\eta} \end{pmatrix} =
\begin{pmatrix} 1 & 0 \\ 0 & 1 \end{pmatrix}.$$
After some calculations it follows
$$a = J^{-1} y_{\eta}, \ \ b = -J^{-1} x_{\eta}, \ \ c = -J^{-1} y_{\xi}, \ \ d = J^{-1} x_{\xi}.$$
From the definition of $e^i$, i.e. $e^i \cdot e_j = \delta_j^i$, $i,j=1,2$,
it follows 
$$e^1 = (a,b)^t = J^{-1} \begin{pmatrix} y_{\eta}  \\ -x_{\eta} \end{pmatrix} \ \ 
{\rm and} \ \ e^2 = (c,d)^t = \begin{pmatrix} -y_{\xi}  \\ x_{\xi} \end{pmatrix},$$
and finally,
$$\frac{\partial(\xi,\eta)}{\partial(x,y)} =
\begin{pmatrix} \xi_{x}   & \xi_{y}  \\ \eta_{x}  & \eta_{y} \end{pmatrix} =
J^{-1} \begin{pmatrix} y_{\eta}   & -x_{\eta}  \\ -y_{\xi}  & x_{\xi} \end{pmatrix}.$$
\end{proof}

\begin{Theorem} \label{EllipticInvertedEq}
The non-linear inverted differential equations corresponding to the Laplacian equations
(\ref{eq:EllipticEq}), with Dirichlet type boundary conditions, are
\begin{equation}
\label{eq:EllipticInvEq}
\begin{array}{l}
L x \equiv \alpha x_{\xi \xi} -2 \beta x_{\xi \eta} + \gamma x_{\eta \eta} = 0, \\
L y \equiv \alpha y_{\xi \xi} -2 \beta y_{\xi \eta} + \gamma y_{\eta \eta} = 0,
\end{array}
\end{equation}
where
$\alpha = x_{\eta} ^2 + y_{\eta} ^2$, $\beta = x_{\xi} x_{\eta} + y_{\xi} y_{\eta}$ and
$\gamma = x_{\xi} ^2 + y_{\xi} ^2$, and with transformed boundary conditions.
\end{Theorem}

\begin{proof}
the equations (\ref{eq:EllipticEq}) are analytically inverted and the calculations 
are carried out in the physical plane (overlapping mesh).

Given a twice differentiable and continuous scalar function $h$ from 
$\Omega = [a,b] \times [c,d] \subset \mathbb{R}^{2}$ into $\mathbb{R}$,
such that $\Delta h = 0$, then

$$\nabla h = e^1 \frac{\partial h}{\partial \xi} + e^2 \frac{\partial h}{\partial \eta}
= e_1 v^1 + e_2 v^2 \ ,$$
\noindent
where
$$v^1 = (e^1 \frac{\partial h}{\partial \xi} + e^2 \frac{\partial h}{\partial \eta}) \cdot e^1
= e^1 \cdot e^1 \frac{\partial h}{\partial \xi} + e^2 \cdot e^1 \frac{\partial h}{\partial \eta} \ ,$$
$$v^2 = (e^1 \frac{\partial h}{\partial \xi} + e^2 \frac{\partial h}{\partial \eta}) \cdot e^2
= e^1 \cdot e^2 \frac{\partial h}{\partial \xi} + e^2 \cdot e^2 \frac{\partial h}{\partial \eta} \ .$$

Calculating the Laplacian of $h$
$$\Delta h = \nabla \cdot (\nabla h) =
e^1 \cdot \frac{\partial}{\partial \xi} (e_1 v^1 + e_2 v^2) +
e^2 \cdot \frac{\partial}{\partial \eta} (e_1 v^1 + e_2 v^2) =$$
$$\frac{\partial}{\partial \xi} v_1 e^1 \cdot e_1 +
\frac{\partial}{\partial \xi} v_2 e^1 \cdot e_2 +
\frac{\partial}{\partial \eta} v_1 e^2 \cdot e_1 +
\frac{\partial}{\partial \eta} v_2 e^2 \cdot e_2 =$$
$$e^1 \cdot e^1 \frac{\partial^2 h}{\partial \xi^2} +
2 e^2 \cdot e^1 \frac{\partial^2 h}{\partial \xi \partial \eta} +
e^2 \cdot e^2 \frac{\partial^2 h}{\partial \eta^2} \ .$$

The terms $e^1 \cdot e^1$, $e^2 \cdot e^2$ and $e^2 \cdot e^1$ are
$$e^1 \cdot e^1 = (J^{-1})^2 (x_{\eta} ^2 + y_{\eta} ^2), \ \
e^2 \cdot e^2 = (J^{-1})^2 (x_{\xi} ^2 + y_{\xi} ^2),$$
$${\rm and} \ \ e^2 \cdot e^1 = -(J^{-1})^2 (x_{\xi} x_{\eta} + y_{\xi} y_{\eta}).$$

Now, defining $\alpha = x_{\eta} ^2 + y_{\eta} ^2$, 
$\beta = x_{\xi} x_{\eta} + y_{\xi} y_{\eta}$ and
$\gamma = x_{\xi} ^2 + y_{\xi} ^2$, and using that $\Delta h = 0$, follows
$$\alpha \frac{\partial^2 h}{\partial \xi^2} -
2 \beta \frac{\partial^2 h}{\partial \xi \partial \eta} +
\gamma \frac{\partial^2 h}{\partial \eta^2} = 0 \ .$$

Finally, the non-linear inverted differential equations 
(\ref{eq:EllipticInvEq}) follows directly, 
for the functions $\xi$ and $\eta$ of the transformation 
$f$ together with equations (\ref{eq:EllipticEq}).
\end{proof}

In the proof of Theorem~\ref{EllipticInvertedEq},
formulas for $e^1 \cdot e^1$, $e^2 \cdot e^2$ and $e^2 \cdot e^1$ were deduced. 
This results together with equations (\ref{eq:inv-matrix_g}), implies that
$$e^1 \cdot e^1 = (J^{-1})^2 (x_{\eta} ^2 + y_{\eta} ^2) = g^{-1} g_{22} = g^{11},$$
$$e^2 \cdot e^2 = (J^{-1})^2 (x_{\xi} ^2 + y_{\xi} ^2) = g^{-1} g_{11} = g^{22},$$
$${\rm and} \ \ e^2 \cdot e^1 = -(J^{-1})^2 (x_{\xi} x_{\eta} + y_{\xi} y_{\eta}) = -g^{-1} g_{21} = g^{21}.$$
Then, $g^{ij} = e^i \cdot e^j$ for $i,j = 1,2$, i.e.
the components of the inverse matrix of the matrix $(g_{ij})$ can be calculated
by inner products of the contravariant base vectors.

\begin{Lemma} \label{pos_definite} 
The equations (\ref{eq:EllipticInvEq}) in Theorem~\ref{EllipticInvertedEq} are elliptic.
\end{Lemma}


\medskip
\noindent
{\it Discretization.}

The equations in (\ref{eq:EllipticInvEq}) are coupled and nonlinear.
The unknowns are the positions of the mesh nodes in the physical plane $x$ and $y$.

\begin{Definition} \label{oper_diferencias} 
Given a twice differentiable and continuous scalar function $x$ from 
$\Omega = [a,b] \times [c,d] \subset \mathbb{R}^{2}$ into $\mathbb{R}$, 
let's define the first and second order finite difference operators of 
$x$ with respect to $\xi$ and $\eta$ as follows
\begin{equation}
\label{eq:DiffFinitas}
\begin{array}{l}
x_{\overline{\xi}}(\xi,\eta) \equiv
\frac{x(\xi + \Delta \xi,\eta)  - x(\xi - \Delta \xi,\eta)}{2 \Delta \xi}, \\
x_{\overline{\eta}}(\xi,\eta) \equiv
\frac{x(\xi,\eta + \Delta \eta)  - x(\xi,\eta  - \Delta \eta)}{2 \Delta \eta}, \\
x_{\overline{\xi \xi}}(\xi,\eta) \equiv
\frac{x(\xi + \Delta \xi,\eta) - 2 x(\xi,\eta) + x(\xi - \Delta \xi,\eta)}{(\Delta \xi) ^2}, \\
x_{\overline{\eta \eta}}(\xi,\eta) \equiv
\frac{x(\xi,\eta + \Delta \eta) - 2 x(\xi,\eta) + x(\xi,\eta  - \Delta \eta)}{(\Delta \eta) ^2}, \\
x_{\overline{\xi \eta}}(\xi,\eta) \equiv 
\frac{x(\xi + \Delta \xi,\eta + \Delta \eta) - x(\xi + \Delta \xi,\eta - \Delta \eta)
+ x(\xi - \Delta \xi,\eta - \Delta \eta) - x(\xi - \Delta \xi,\eta + \Delta \eta)}
{4 \; \Delta \xi \Delta \eta}.
\end{array}
\end{equation}
where $\Delta \xi$ and $\Delta \eta$ are the constant spacings in the $\xi$ and $\eta$ 
directions respectively.
\end{Definition}

\begin{Lemma} \label{DF-approx-error}
Given a twice differentiable and continuous scalar function $x$ from 
$\Omega = [a,b] \times [c,d] \subset \mathbb{R}^{2}$ into $\mathbb{R}$,
the approximation errors for the first and second derivatives of $x$,
using the finite difference operators (\ref{eq:DiffFinitas}) 
given in definition~\ref{oper_diferencias}, are order 2.
\end{Lemma}


\begin{Remark}
Lemmas \ref{pos_definite} and \ref{DF-approx-error} are
sufficiently known results, and their proofs are straightforward.
\end{Remark}


\begin{Lemma} \label{Lx-approx-error}
Given a twice differentiable and continuous scalar function $x$ from 
$\Omega = [a,b] \times [c,d] \subset \mathbb{R}^{2}$ into $\mathbb{R}$,
the operator
$$L x \equiv \alpha x_{\xi \xi} -2 \beta x_{\xi \eta} + \gamma x_{\eta \eta}$$
can be approximated by
$$L_h x \equiv \alpha x_{\overline{\xi \xi}} -2 \beta x_{\overline{\xi \eta}} 
+ \gamma x_{\overline{\eta \eta}}$$
with an approximation error of order 2,
where $x_{\overline{\xi \xi}}$, $x_{\overline{\xi \eta}}$ and $x_{\overline{\eta \eta}}$
are the finite differences (\ref{eq:DiffFinitas}) in definition~\ref{oper_diferencias}.
\end{Lemma}

\begin{proof}
the proof follows from a straightforward calculation
using the results in lemma~\ref{DF-approx-error}
$$L_h x - L x = \alpha (x_{\overline{\xi \xi}} - x_{\xi \xi}) 
- 2 \beta (x_{\overline{\xi \eta}} - x_{\xi \eta}) + \gamma (x_{\overline{\eta \eta}} - x_{\eta \eta})
= O((\Delta \xi)^2 + (\Delta \eta)^2). \qedhere$$
\end{proof}

\begin{Theorem} \label{DiscreteEllipticInvertedEq}
The discrete version of the system of equations (\ref{eq:EllipticInvEq}), 
for any vertex $(i,j)$ of the mesh in the physical plane, can be written as follows
\begin{equation}
\label{eq:EllipticDisc-1}
\begin{array}{l}
2(x_{i+1,j} -2x_{i,j} + x_{i-1,j}) \cdot ((x_{i,j+1} - x_{i,j-1})^2 + (y_{i,j+1} - y_{i,j-1})^2) - \\
  -(x_{i+1,j+1} - x_{i+1,j-1} + x_{i-1,j-1} - x_{i-1,j+1}) \cdot  \\
  \cdot ((x_{i+1,j} - x_{i-1,j}) \cdot (x_{i,j+1} - x_{i,j-1}) + \\
   + (y_{i+1,j} - y_{i-1,j}) \cdot (y_{i,j+1} - y_{i,j-1})) +  \\
   +2(x_{i,j+1} -2x_{i,j} + x_{i,j-1}) \cdot \\
  \cdot ((x_{i+1,j} - x_{i-1,j})^2 + (y_{i+1,j} - y_{i-1,j})^2) = 0,
\end{array}
\end{equation}
\begin{equation}
\label{eq:EllipticDisc-2}
\begin{array}{l}
2(y_{i+1,j} -2y_{i,j} + y_{i-1,j}) \cdot ((x_{i,j+1} - x_{i,j-1})^2 + (y_{i,j+1} - y_{i,j-1})^2) - \\
  -(y_{i+1,j+1} - y_{i+1,j-1} + y_{i-1,j-1} - y_{i-1,j+1}) \cdot \\
  \cdot ((x_{i+1,j} - x_{i-1,j}) \cdot (x_{i,j+1} - x_{i,j-1}) + \\
  + (y_{i+1,j} - y_{i-1,j}) \cdot (y_{i,j+1} - y_{i,j-1})) + \\
  +2(y_{i,j+1} -2y_{i,j} + y_{i,j-1}) \cdot ((x_{i+1,j} - x_{i-1,j})^2 + (y_{i+1,j} - y_{i-1,j})^2) = 0.
\end{array}
\end{equation}
\end{Theorem}

\begin{proof}
the finite differences (\ref{eq:DiffFinitas}) given in definition~\ref{oper_diferencias} 
are used to approximate the derivatives appearing
in the set of elliptic equations (\ref{eq:EllipticInvEq}) in the physical plane.

After some calculations, the discrete version (\ref{eq:EllipticDisc-1}-\ref{eq:EllipticDisc-2}) 
of the system of equations (\ref{eq:EllipticInvEq}) is obtained.
\end{proof}

The nonlinear system (\ref{eq:EllipticDisc-1}-\ref{eq:EllipticDisc-2}) 
consists of $2mn$ equations, where $m$ and $n$ are the 
dimensions of the initial cartesian grid.
In this nonlinear system, the coordinates of the vertices on the internal and external boundaries are
fixed. If $p$ is the number of fixed vertices, then the final nonlinear system has $r=2(mn-p)$ equations
and unknowns. These unknowns correspond to the coordinates of the overlapping mesh nodes which are not
fixed (i.e. those which are not vertices).

\section{Algorithms} \label{QGG-algoritmo}


The general process for generating the mesh honoring the internal boundaries,
given in section \ref{QGG-metodologia}, can be translated into the following algorithm.
This algorithm was implemented in fortran 90.

\begin{enumerate}
\setlength\itemsep{0em}
\item Decompose the QIAC and IQIAC into a set of SIAC
\item Extend the IAC to the external boundaries to form SIAC
\item Generate the initial uniform cartesian grid
\item For each SIAC (including the original SIAC)
  \begin{enumerate}
  \setlength\itemsep{0em}
  \item Associate a line from the initial grid
  \item Redistribute the nodes on the associated line into the SIAC
  \end{enumerate}
\item Redistribute nodes on the external boundaries
\item Smooth the overlapping mesh 
\end{enumerate}

In item 4 it is important to point out that the treatment is slightly different for
SIAC coming from QIAC or IQIAC. In this case, these QIAC or IQIAC are first enclosed
in their minimal rectangles such that the processes of line association and redistribution
of nodes happen locally. The overlapping local grid is then connected logically to the
initial cartesian grid.

The smooth mesh is obtained by solving the set of elliptic equations (\ref{eq:EllipticInvEq})
presented in section 3. As mentioned before, this set of elliptic equations
is approximated by finite differences, conducing to the nonlinear system 
(\ref{eq:EllipticDisc-1}-\ref{eq:EllipticDisc-2}).

This system of nonlinear equations $F(v)=0$, with $F:\mathbf{R}^{r}\rightarrow \mathbf{R}^{r}$ 
and the vector $v$ representing the final coordinates of each node of the grid,
is solved numerically by SANE
(Spectral Approach for Nonlinear Equations) method (see \citep{LaCruzRaydan2003}).
The methodology is based on minimization techniques without restrictions, where the
objective function involved is $g(v)=F(v)^{t}F(v)$.
The key issue in this methodology is the combination between the search direction
($d_{k}=\pm F(x_{k})$) and the fact that the nonmonotone line search involves a
quadratic interpolation. The Jacobian matrix of $F$ is, in general, non symmetric and
large.

The SANE method, implemented in fortran 90, is presented in SANE Algorithm.
It is important to remark that this algorithm only involves one smooth process,
compared to previous works (see 
\citep{Hyman2000}, \citep{Borregales_etal2009} and \citep{Valido_etal2012}).

\noindent \begin{center}
\textbf{SANE Algorithm} (see \citep{LaCruzRaydan2003}).

\begin{tabular}{ll}
Let & $\alpha_{0} \in \mathbf{R}$, $M\geq0$, $\gamma>0$ and
$0<\sigma_{1}<\sigma_{2}<1$, $0<\epsilon<1$\tabularnewline
    & and $\delta\in[\epsilon,\frac{1}{\epsilon}]$.\tabularnewline
Let & $v_{0}\in \mathbf{R}^{r}$ be the initial guess and set $k=0$.\tabularnewline
Step 1: & If $\left\Vert F_{k}\right\Vert =0$, stop the process.\tabularnewline
Step 2: & If $\frac{|F_{k}^{t}\: J_{k}\: F_{k}|}{F_{k}^{t}\: F_{k}}<\epsilon$,
stop the process.\tabularnewline
Step 3: & If $\alpha_{k}\leq\epsilon$ or $\alpha_{k}\geq\frac{1}{\epsilon},$
then set $\alpha_{k}=\delta$.\tabularnewline
Step 4: & Set $sgn_{k}=sgn(F_{k}^{t}\: J_{k}\: F_{k})$ and 
  $d_{k}=-sgn_{k}\: F_{k}$.\tabularnewline
Step 5: & Set $\lambda=\frac{1}{\alpha_{k}}$.\tabularnewline
Step 6: & If $f(v_{k}+\lambda d_{k})\leq\max_{0\leq j\leq\min\{k,M\}}f(v_{k-j})+
  2\gamma\lambda\: F_{k}^{t}\: J_{k}\: d_{k}$,\tabularnewline
        & go to Step 8.\tabularnewline
Step 7: & Choose $\sigma\in[\sigma_{1},\sigma_{2}]$, and set $\lambda=\sigma\lambda$.
  Now go to Step 6.\tabularnewline
Step 8: & Set $\lambda_{k}=\lambda$, $v_{k+1}=v_{k}+\lambda_{k}\: d_{k}$,
  $w_{k}=F_{k+1}-F_{k}$.\tabularnewline
Step 9: & Set $\alpha_{k+1}=sgn_{k}\left(\frac{d_{k}^{t}\: w_{k}}{\lambda_{k}\: d_{k}^{t}\: d_{k}}\right)$,
  $k=k+1$ and go to Step 1.\tabularnewline
\end{tabular}
\par\end{center}

\section{Examples and numerical results} \label{QGG-ejemplos}

Examples 1 to 4 presented in this section correspond to rectangular oil reservoirs.
The internal boundaries to be considered are SIAC, QIAC and IQIAC and combinations
of these.
Example 1 (see Fig.\ref{Example1}) includes two horizontal and one vertical SIAC,
example 2 (see Fig.\ref{Example2}) consists of one QIAC,
example 3 (see Fig.\ref{Example3}) contains three QIAC, and
example 4 (see Fig.\ref{Example4}) includes three SIAC, one QIAC, and one IQIAC.
Table 1 shows the results obtained with SANE and the Newton-GMRES methodologies.

In Table 1, $||F||_{2}$ is the 2-norm of the function $F$, and 
$||M_{1}-M_{2}||_{\infty}$ corresponds to the 
infinite norm of the vector difference of the coordinates of the meshes 
obtained with both methodologies.
$L$ represents the size of the domain (side length of the domain, in all cases
the domain is a square).
Additionally, N-GMRES means Newton's method, nested with a GMRES type method.

\noindent \begin{center}
\textbf{Table 1}. Comparison of SANE and Newton-GMRES methods.\smallskip{}

\begin{tabular}{|c|c|c|c|c|}
\hline 
Ex & Methods & $||F||_{2}$ & $\frac{||M_{1}-M_{2}||_{\infty}}{L}$ & Time\tabularnewline
         &         &             &                            & Iterations\tabularnewline
\hline
\hline  
Ex 1 & N-GMRES & $2.0425\times10^{-1}$ & $0.064585$ & $0.9840$ sec\tabularnewline
          &         &                       &                       & 8N / 8 GMRES\tabularnewline
\cline{2-3} \cline{5-5} 
  & SANE & $2.0430\times10^{-1}$ &  & $0.1100$ sec\tabularnewline
  &      &                       &  & 185\tabularnewline
\hline 
Ex 2 & N-GMRES & $2.4367\times10^{-2}$ & $0.0082315$ & $0.5310$ sec\tabularnewline
          &         &                       &                       & 6N / 6 GMRES\tabularnewline
\cline{2-3} \cline{5-5} 
  & SANE & $2.4414\times10^{-2}$ &  & $0.0310$ sec\tabularnewline
  &      &                       &  & 107\tabularnewline
\hline 
Ex 3 & N-GMRES & $6.2817\times10^{-1}$ & $0.055343$ & $1.2970$ sec\tabularnewline
          &         &                       &                       & 24N / 24 GMRES\tabularnewline
\cline{2-3} \cline{5-5} 
  & SANE & $5.3339\times10^{-1}$ &  & $0.0780$ sec\tabularnewline
  &      &                       &  & 93\tabularnewline
\hline 
Ex 4 & N-GMRES & $2.6808\times10^{-3}$ & $1.16205$ & $54.6710$ sec\tabularnewline
          &         &                       &          & 55N / 55 GMRES\tabularnewline
\cline{2-3} \cline{5-5} 
  & SANE & $2.8194\times10^{-3}$ &  & $0.4530$ sec\tabularnewline
  &      &                       &  & 152\tabularnewline
\hline
\end{tabular}
\par\end{center}

Example 5 combines the use of QIAC with SIAC to model a channel of high permeability 
and the drainage areas of two producer wells in the vicinity of the channel 
(see Fig.\ref{Example5}).

Example 6 represents a circumference shaped obstacle within a reservoir 
(see Fig.\ref{Example6}) where the permeability could be low compared to 
the other areas of $\Omega$.

In both cases it is observed how good the mesh adapts to the internal boundary.

Buitrago et al. used this type of meshes to have a good representation of the
internal structures of the reservoir and to solve the convection diffusion equation
in two dimensions based on finite volume methods (see \citep{Buitragoetal2015}).

\begin{figure}[ht]
\noindent \begin{centering}
{\includegraphics[width=2.2in,height=2.2in]{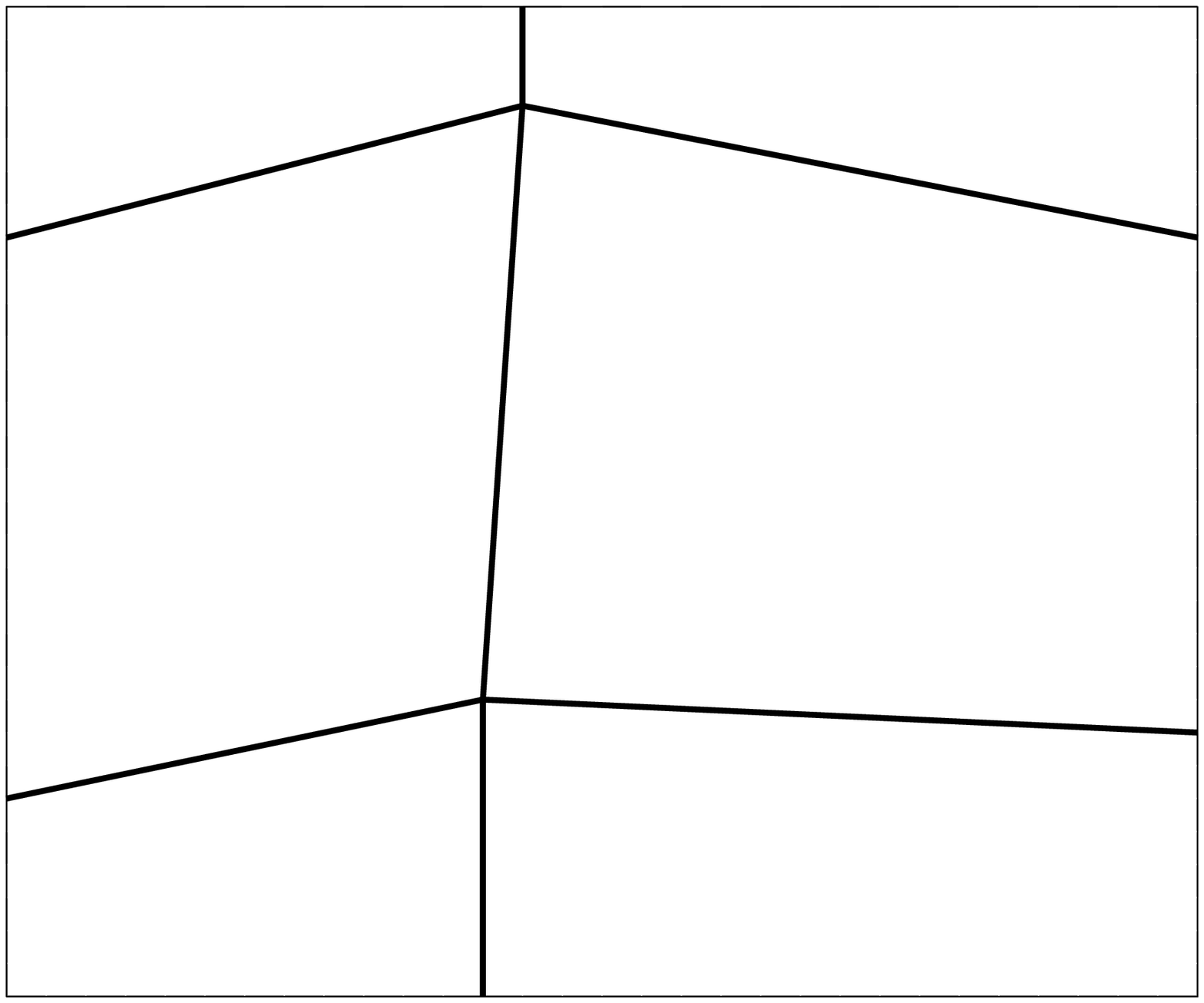}
\includegraphics[width=2.2in,height=2.2in]{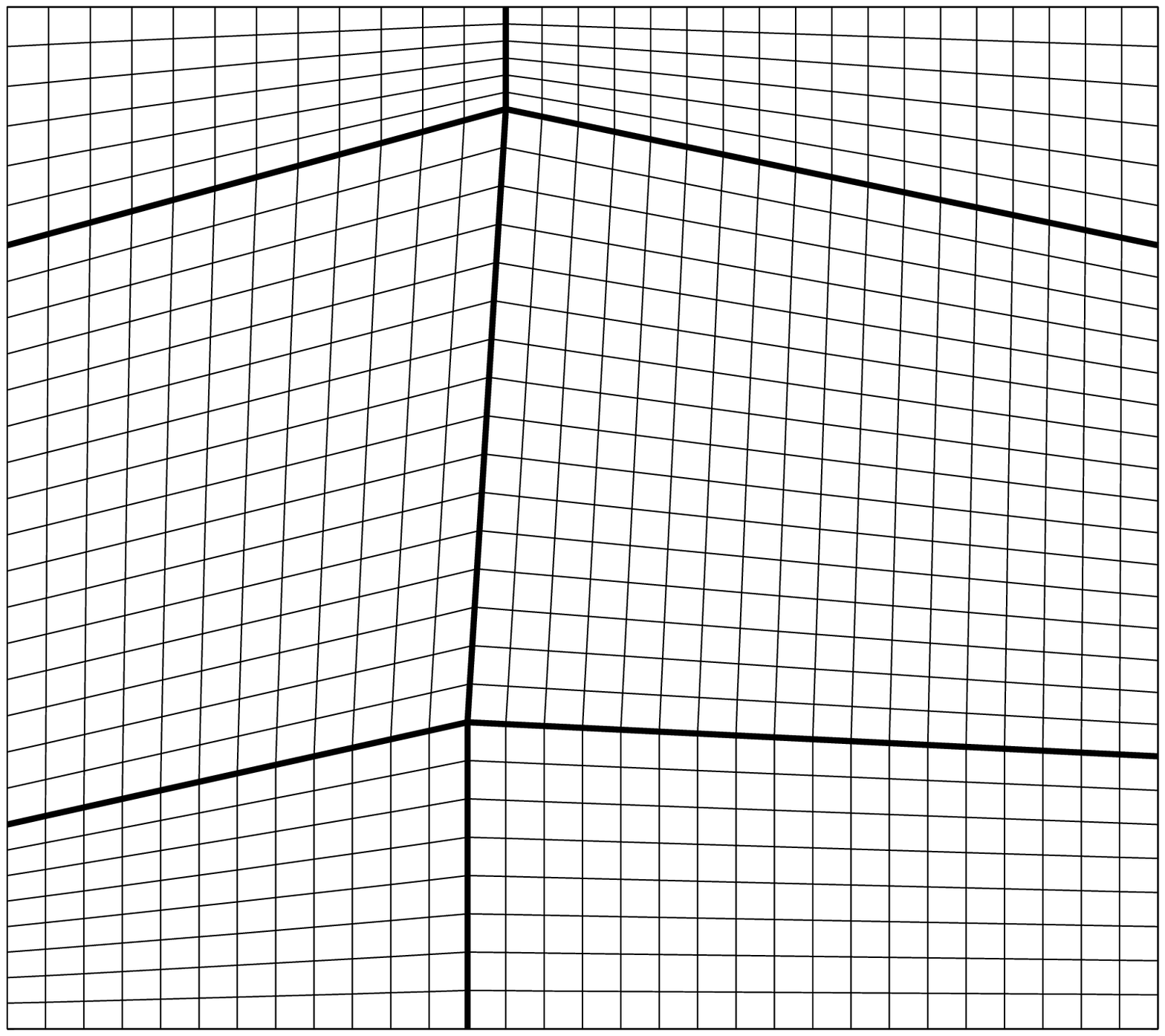}}
\par\end{centering}

\noindent \centering{}\caption{Example 1 with three SIAC, two horizontal and one vertical.
On the right: the smooth mesh is supported on the internal boundaries.}
\label{Example1}
\end{figure}

\begin{figure}[ht]
\noindent \begin{centering}
{\includegraphics[width=2.2in,height=2.2in]{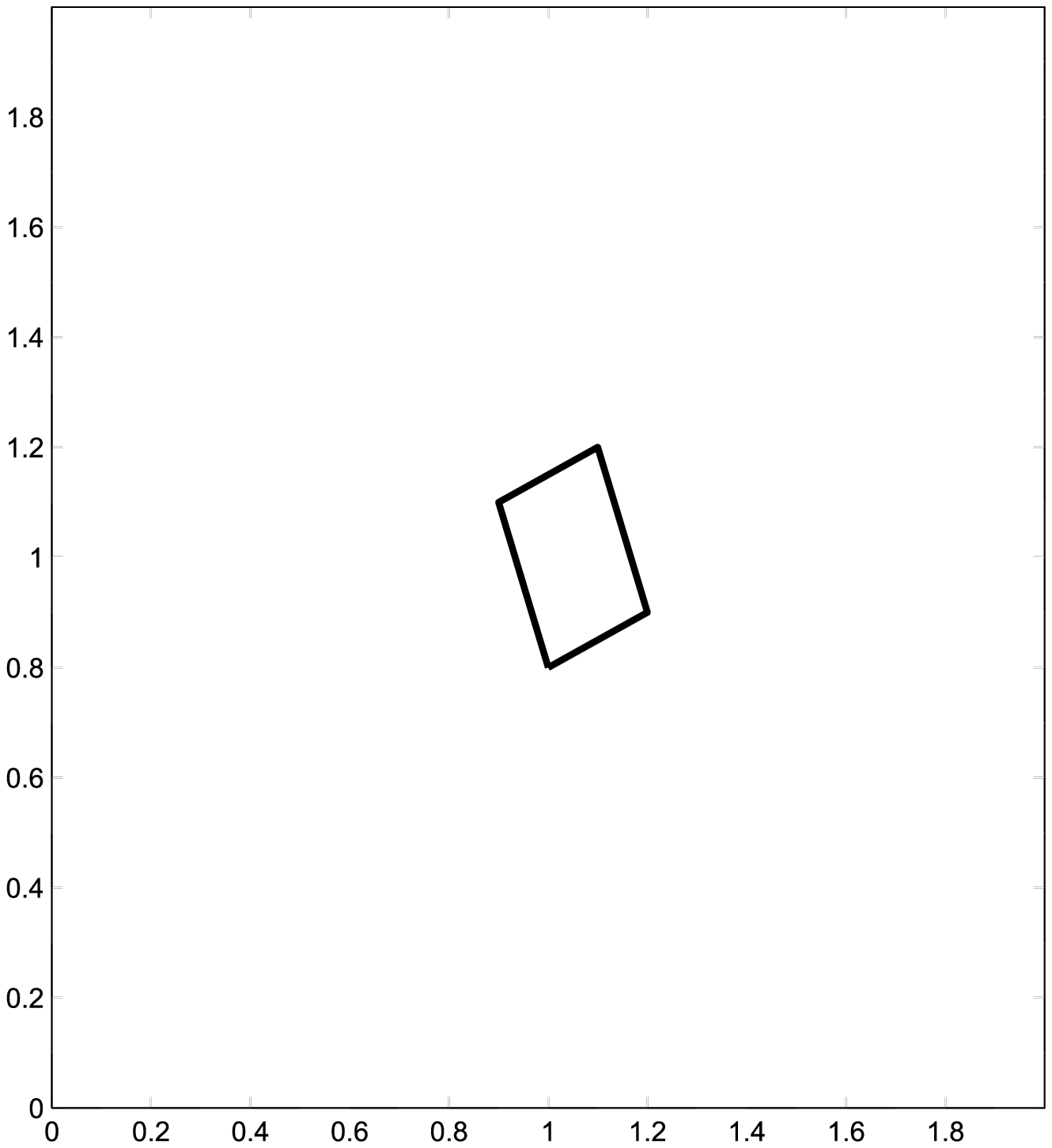}
\includegraphics[width=2.2in,height=2.2in]{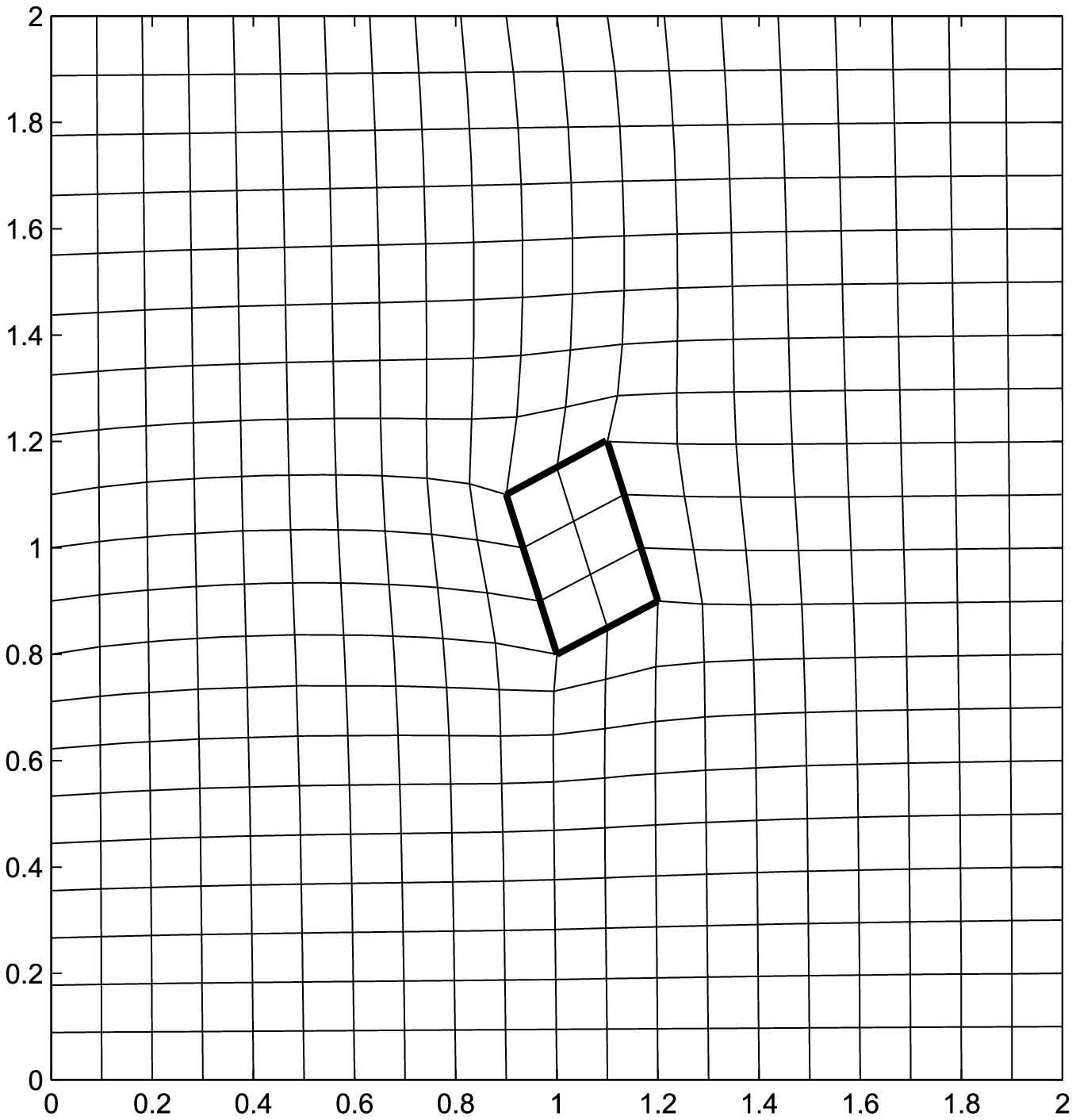}}
\par\end{centering}

\noindent \centering{}\caption{Example 2 with one QIAC.
On the right: the smooth mesh is supported on the internal boundaries.}
\label{Example2}
\end{figure}

\begin{figure}[ht]
\noindent \begin{centering}
{\includegraphics[width=2.2in,height=2.2in]{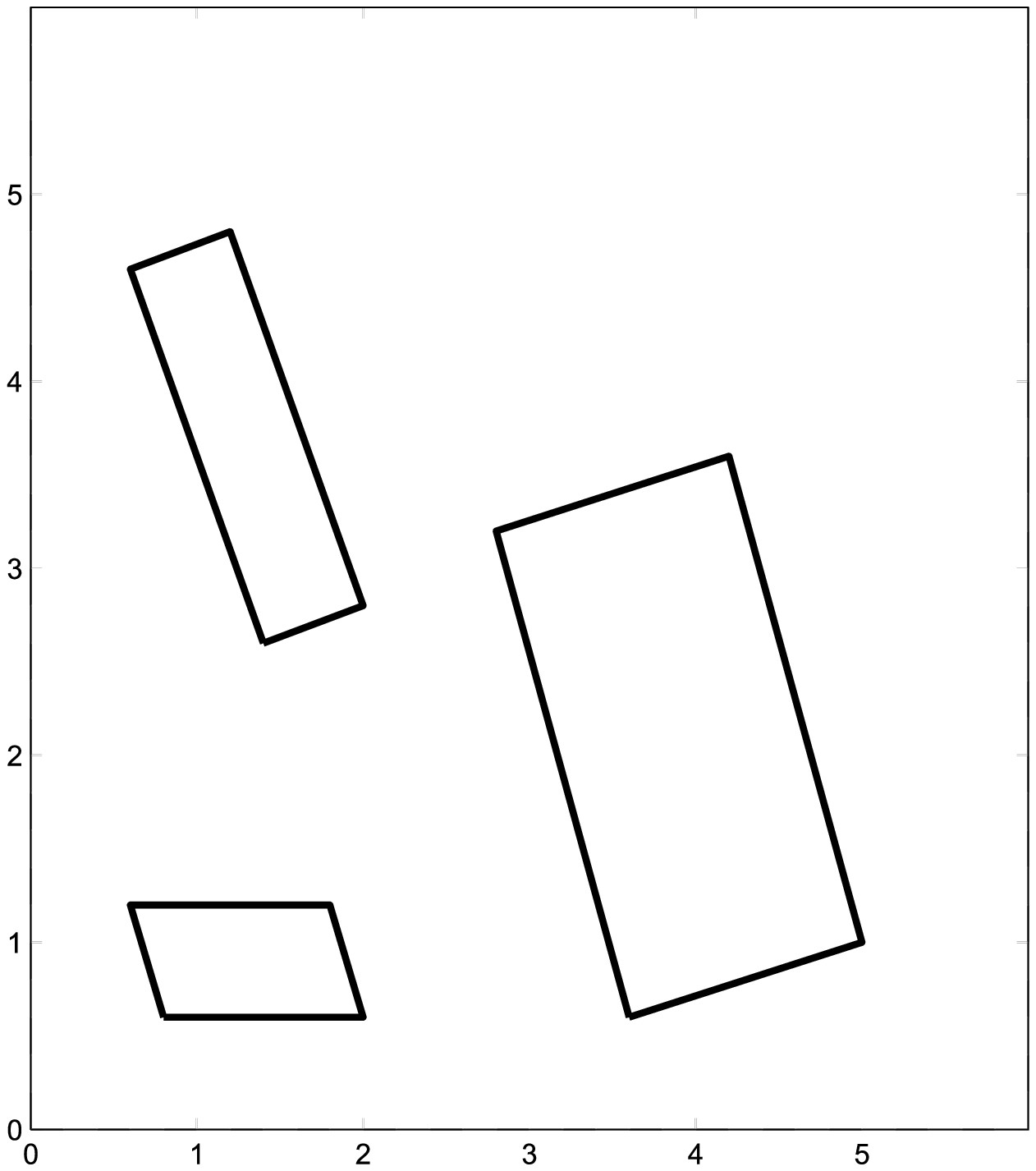}
\includegraphics[width=2.2in,height=2.2in]{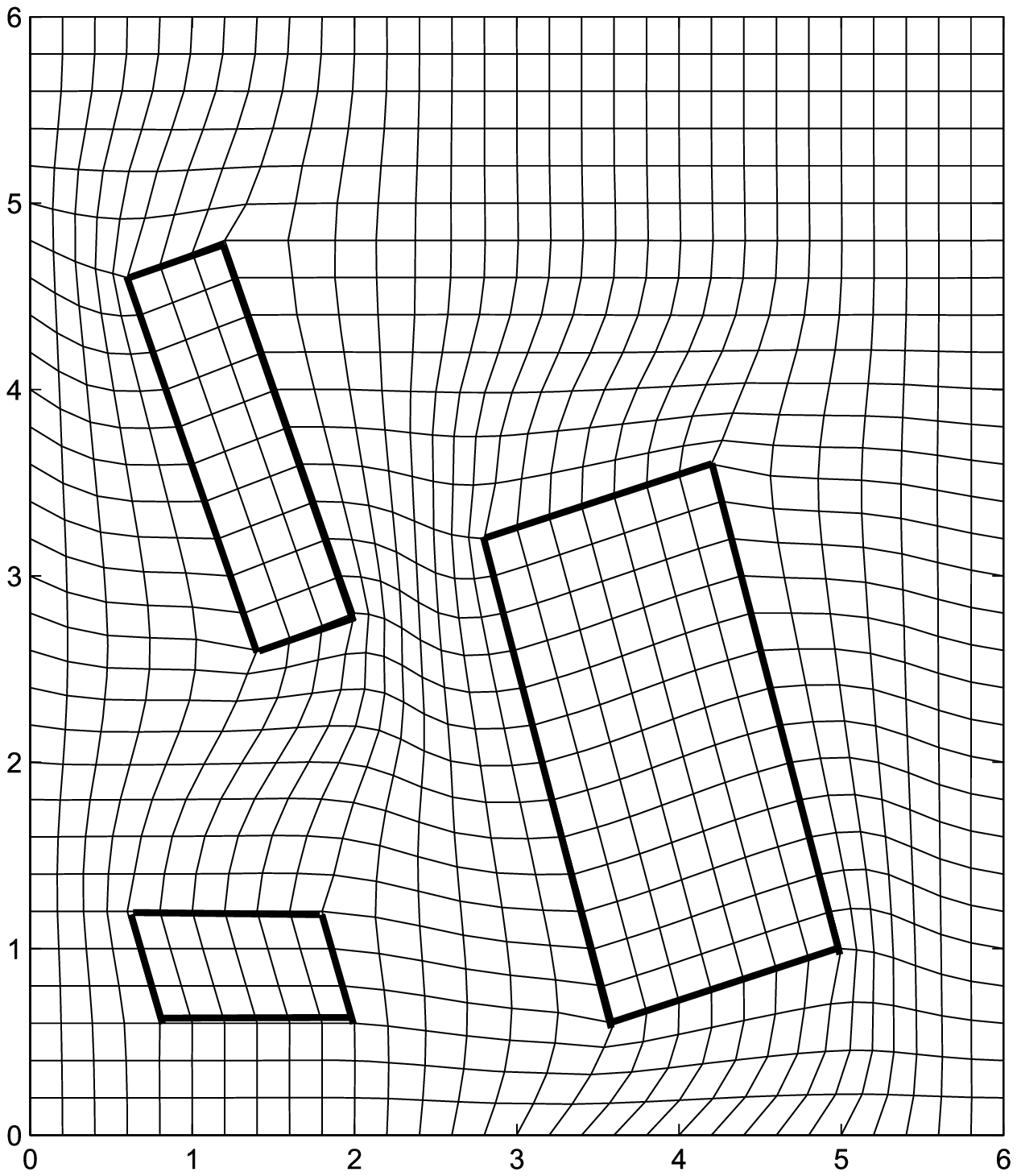}}
\par\end{centering}

\noindent \centering{}\caption{Example 3 with three QIAC.
On the right: the smooth mesh is supported on the internal boundaries.}
\label{Example3}
\end{figure}

\begin{figure}[ht]
\noindent \begin{centering}
{\includegraphics[width=2.2in,height=2.2in]{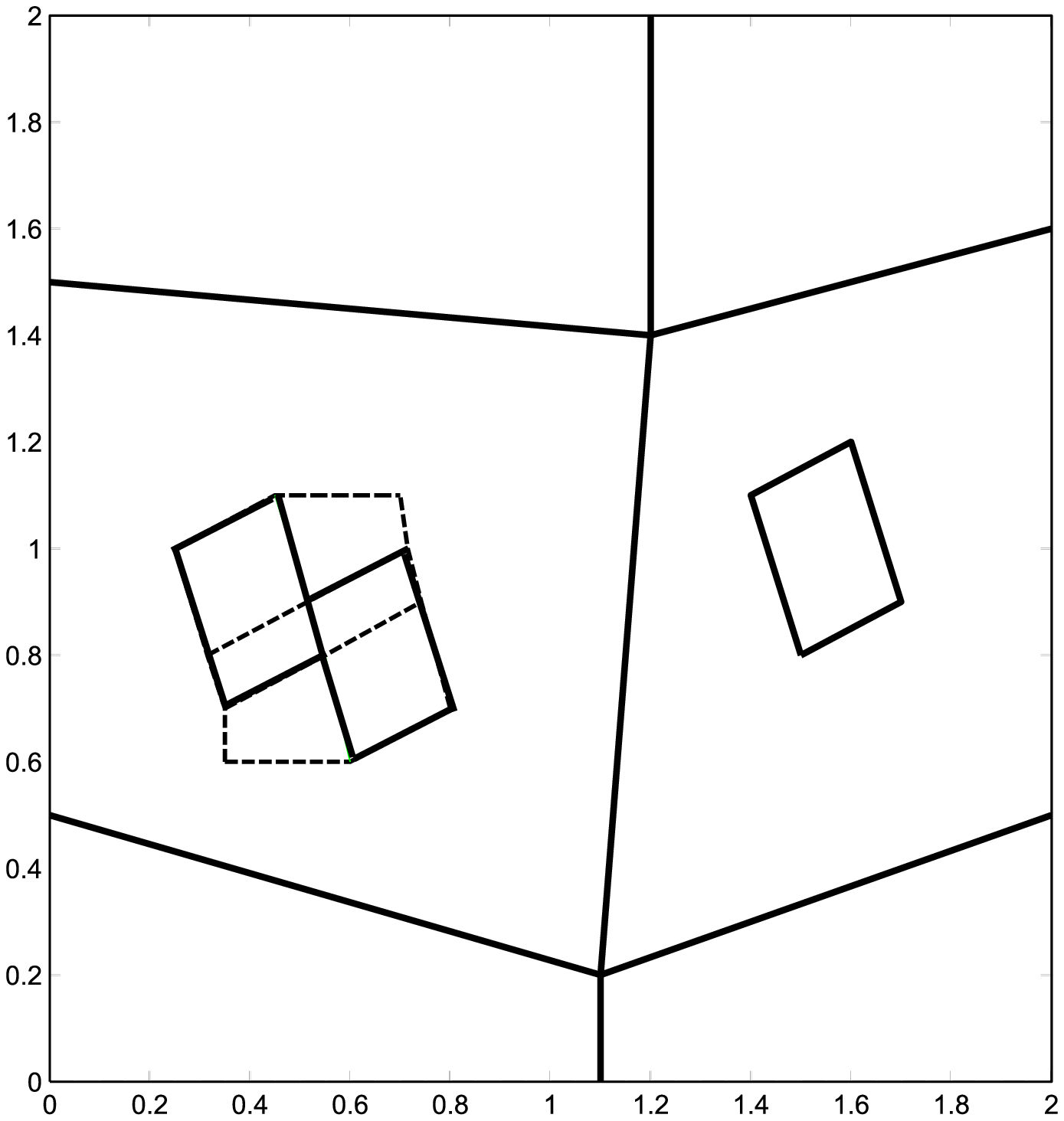}
\includegraphics[width=2.2in,height=2.2in]{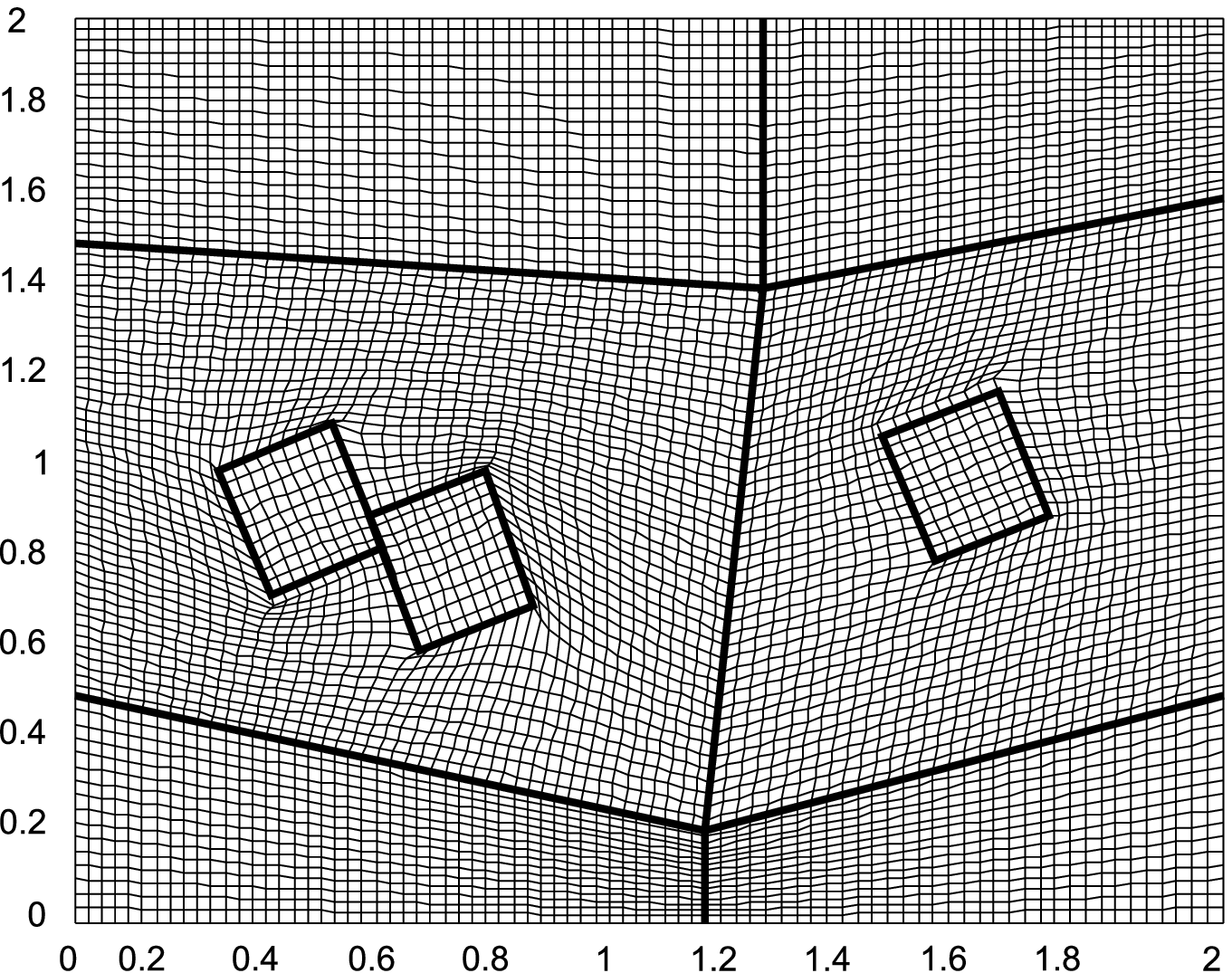}}
\par\end{centering}

\noindent \centering{}\caption{Example 4 with three SIAC, one QIAC, and one IQIAC.
On the right: the smooth mesh is supported on the internal boundaries.}
\label{Example4}
\end{figure}

\begin{figure}[ht]
\noindent \begin{centering}
{\includegraphics[width=2.2in,height=2.2in]{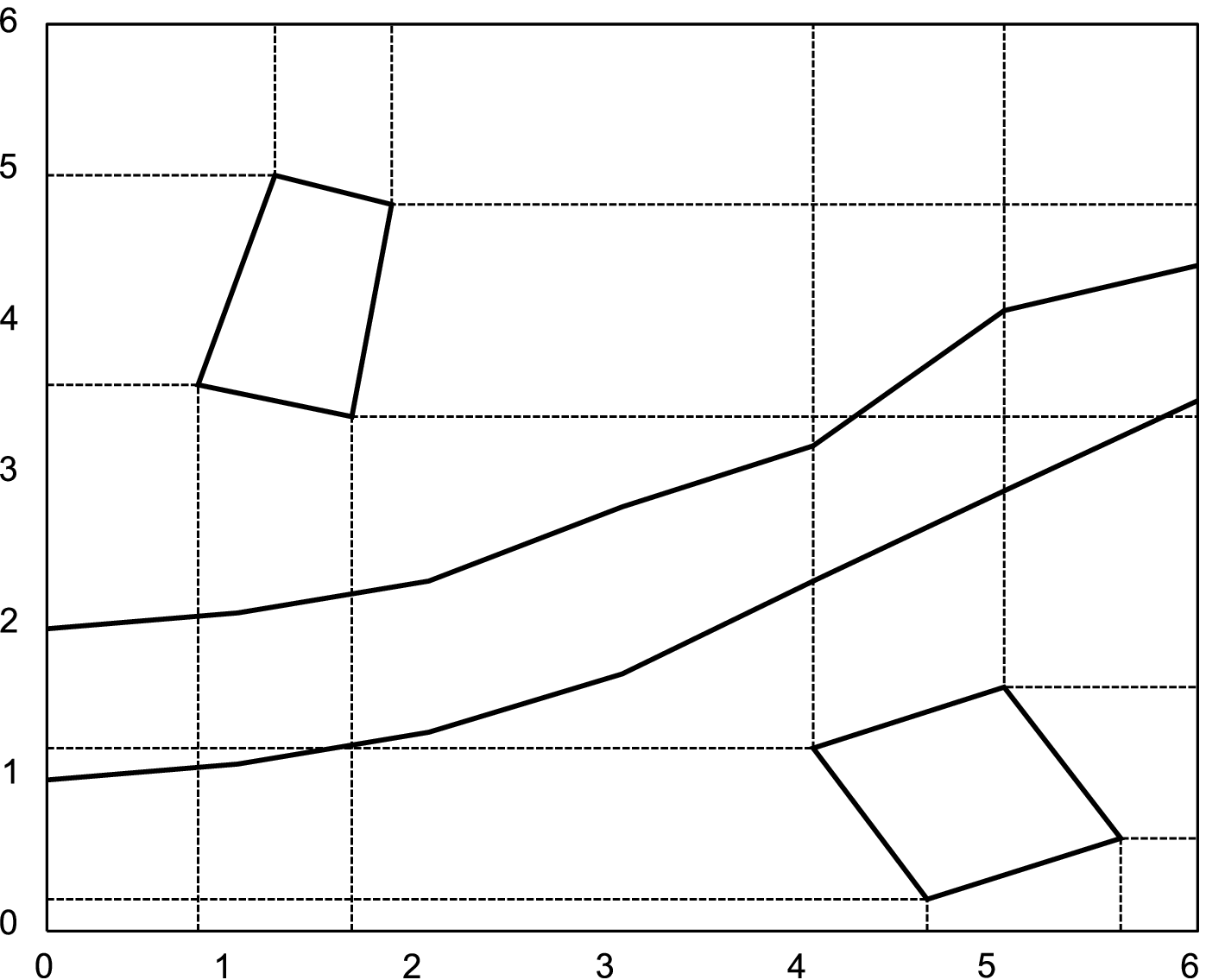}
\includegraphics[width=2.2in,height=2.2in]{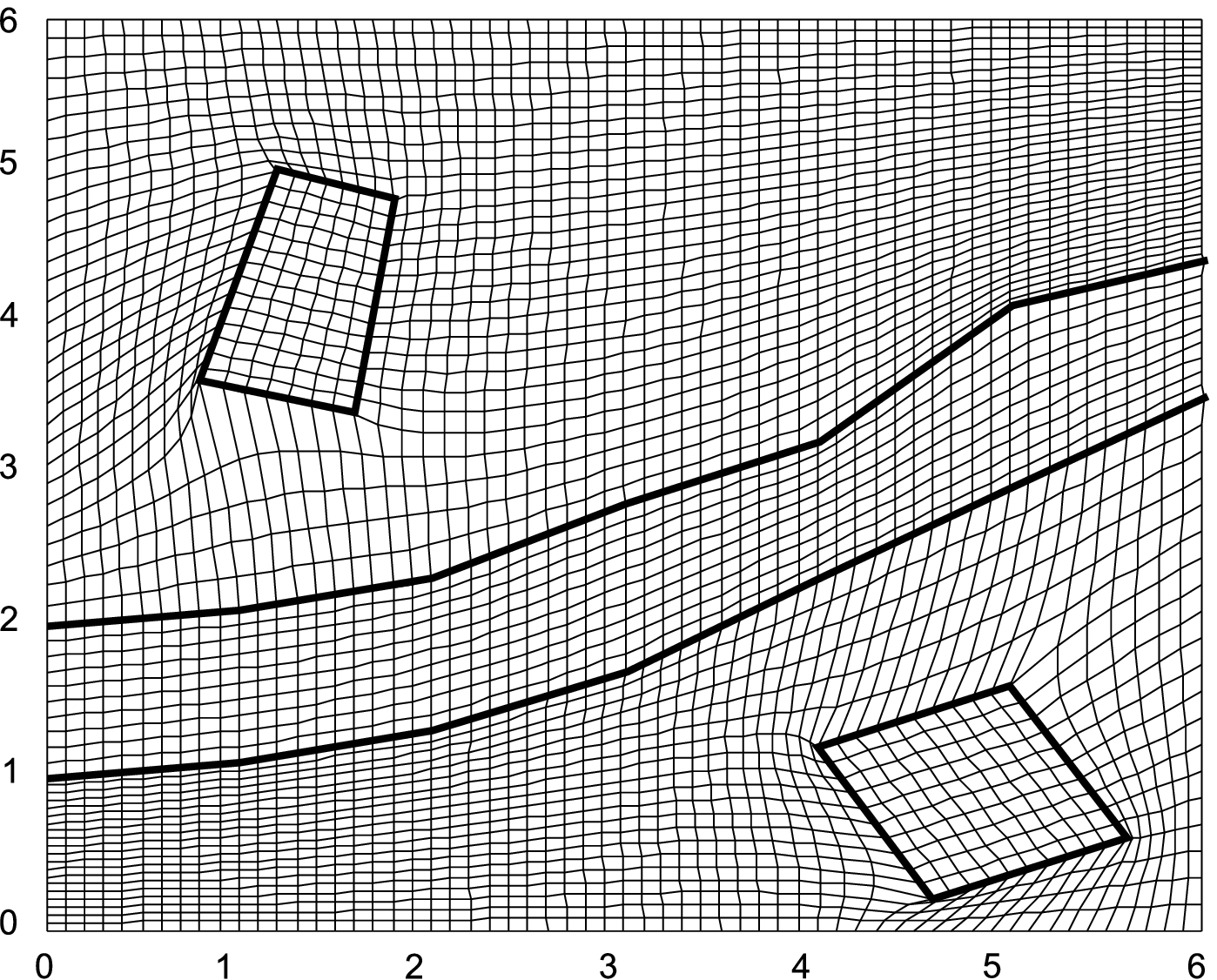}}
\par\end{centering}

\noindent \centering{}\caption{Example 5 with one channel and two QIAC.
On the right: the smooth mesh is supported on the internal boundaries.
This example models a channel of high permeability in a reservoir
and the drainage areas of two producer wells in the vicinity of the channel.}
\label{Example5}
\end{figure}

\begin{figure}[ht]
\noindent \begin{centering}
{\includegraphics[width=2.2in,height=2.2in]{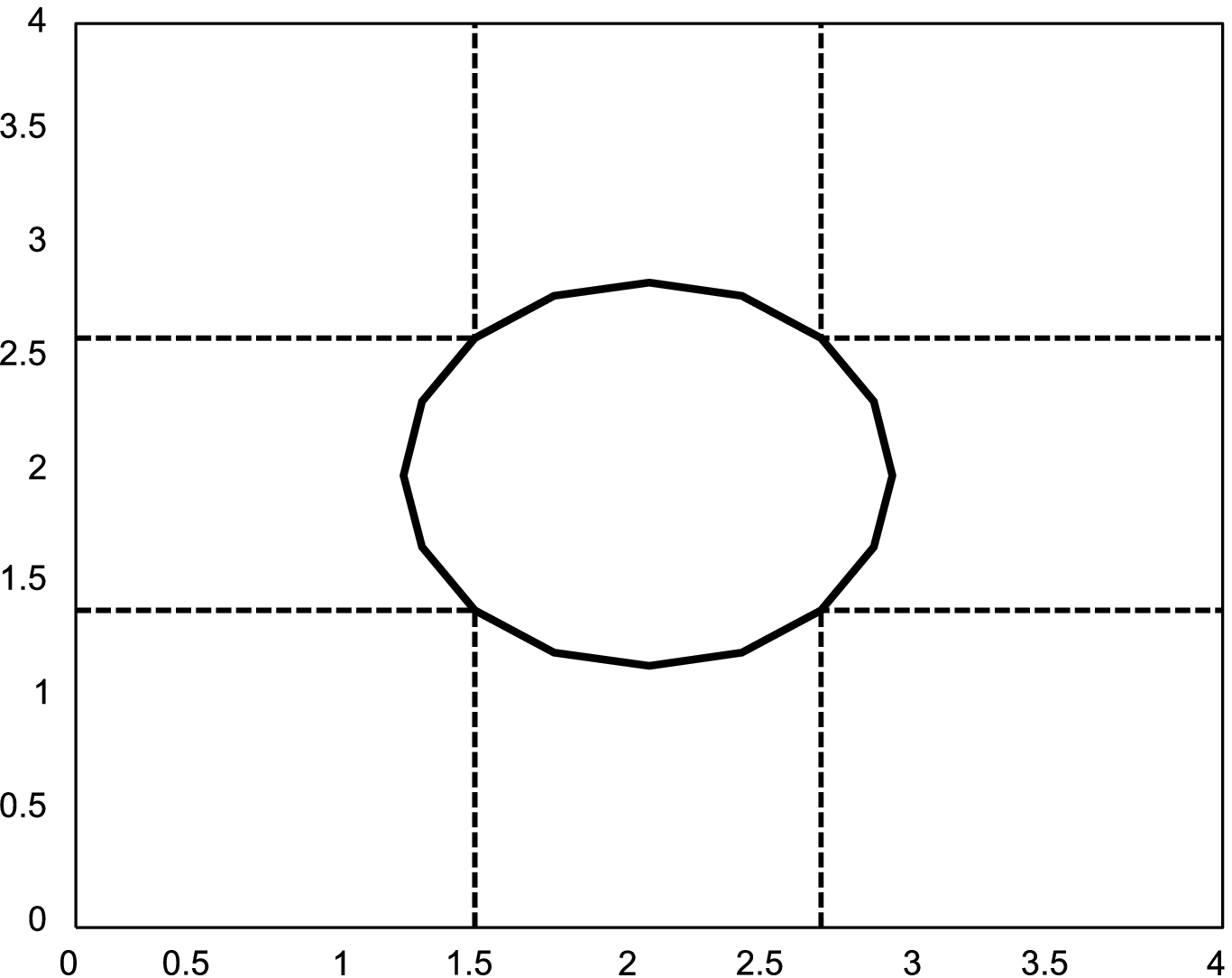}
\includegraphics[width=2.2in,height=2.2in]{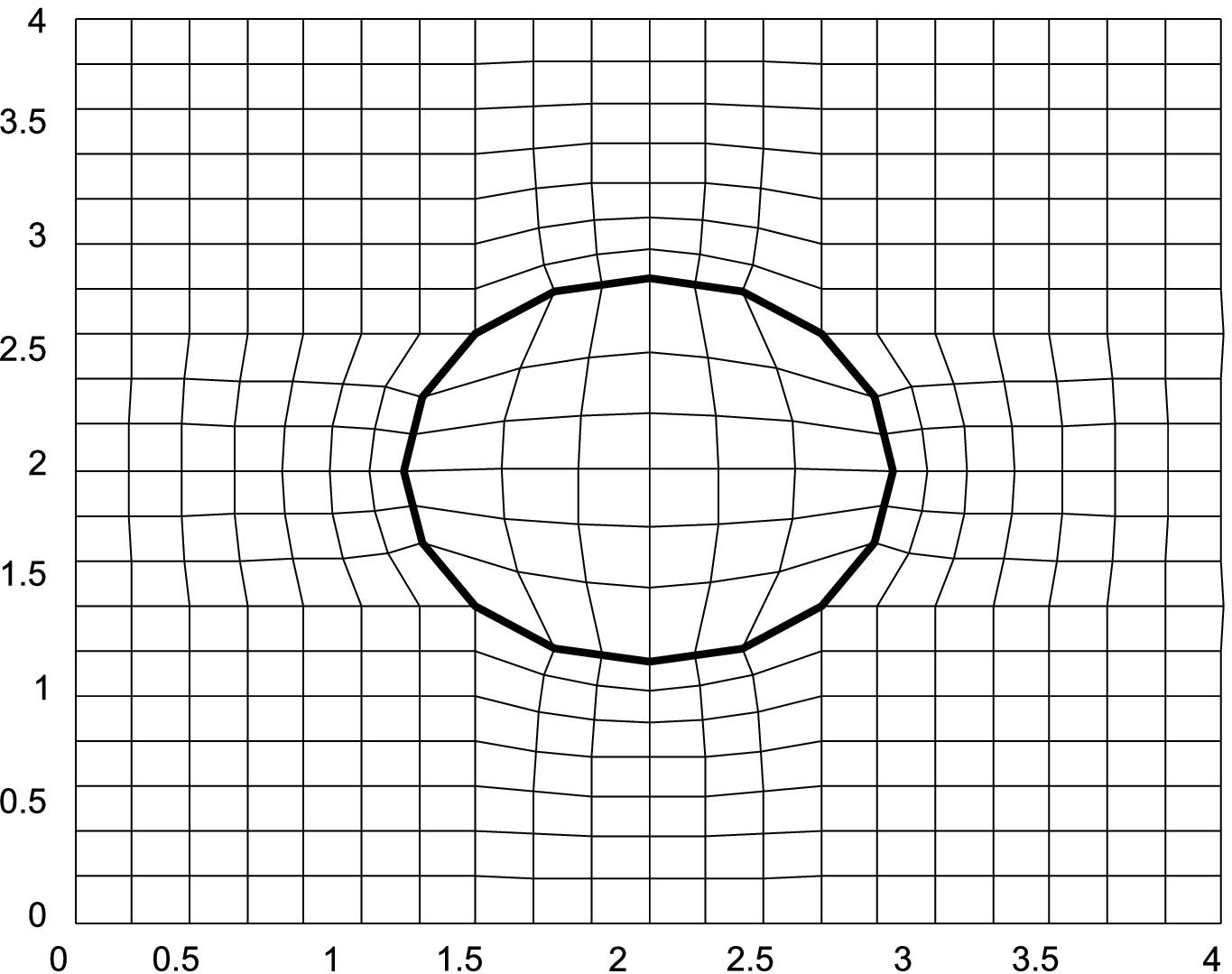}}
\par\end{centering}

\noindent \centering{}\caption{Example 6 with one circumference.
On the right: the smooth mesh is supported on the internal boundaries.
This example models a very low permeability obstacle within a reservoir.}
\label{Example6}
\end{figure}

\section{Conclusions} \label{QGG-conclusiones}

\begin{enumerate}
\item The physical domain was discretized using non orthogonal structured (valence 4) grid 
consisting only of quadrilaterals honoring the internal structures of a reservoir.
\item The methodology does not need for previous step where triangles are generated.
\item The SANE algorithm is a robust option for solving the nonlinear system 
associated to the methodology for generation of 2D quadrilateral meshes 
adapted to internal boundaries.
This does not involve nesting of iterative methods to solve the nonlinear system.
\item The mesh generation methodology involves a single smoothing process
compared to previous works.
\item The methodology was implemented in fortran 90, taking advantage of
the potential of the language to create structures adapted to the 
raised geometric problem.
\item Additionally, we suggest to explore the possibility of implementing a code
that exploits the implicit parallelism in the methodology proposed in this paper.
Particularly, the distribution of work that arises naturally, is to
process each internal boundary on each processor of a cluster 
(parallel architecture computer).
\end{enumerate}

\section*{References}

\bibliography{QGG-BuitragoJimenez}

\end{document}